\documentclass{article}
\usepackage[utf8]{inputenc}

\usepackage{amssymb}
\usepackage{amsthm}
\usepackage{amsmath,amscd}
\usepackage[mathscr]{euscript}
\usepackage[all]{xy}
\usepackage{lmodern}
\usepackage[T1]{fontenc}
\usepackage[textwidth=14cm,hcentering]{geometry}
\usepackage[colorlinks=true,linkcolor=red,citecolor=blue]{hyperref}
\usepackage{cleveref}
\usepackage{mathtools}
\usepackage{tikz}
\usetikzlibrary{cd}
\usetikzlibrary{arrows,positioning}
\crefname{ex}{Example}{Examples}

\title{Comparison of Motivic Homotopy Theories}
\author{Tianjian Tan}
\date{}

\newtheorem{thm}{Theorem}[section]
\newtheorem{lm}[thm]{Lemma}
\newtheorem{prop}[thm]{Proposition}

\newtheorem*{thm*}{Theorem}

\theoremstyle{definition}
\newtheorem{defn}[thm]{Definition}

\newtheorem{nota}[thm]{Notation}

\newtheorem{ex}[thm]{Example}

\newtheorem{rmk}[thm]{Remark}

\crefname{ques}{Question}{Questions}

\newcommand{\on}{\operatorname}

\newcommand{\PrL}{\mathsf{Pr}^\mathrm{L} }

\DeclareFontFamily{U}{min}{}
\DeclareFontShape{U}{min}{m}{n}{<-> udmj30}{}

\begin{document}
\maketitle
\begin{abstract}
 We construct a comparison functor from the dual category of motivic homotopy category $\mathcal{SH}$ to the category of $\mathbb{A}^1$-invariant localizing motives $\on{Mot}_{\on{loc}}^{\mathbb{A}^1}$ in the sense of Blumberg, Gepner and Tabuada (with $\mathbb{A}^1$-invariance imposed). We as well construct its non-$\mathbb{A}^1$-invariant analogue: a functor from the dual category of Annala-Iwasa-Hoyois's non-$\mathbb{A}^1$-invariant motivic homotopy category $\mathcal{MS}$ to $\on{Mot}_{\on{loc}}$. After the Barr-Beck argument, these functors factor through categories of modules over a dual version of ($\mathbb{A}^1$-invariant) K-theory spectrum $\on{KGL}^{(\mathbb{A}^1)}$. Over a field that admits resolution of singularities, we show that the $\mathbb{A}^1$-invariant factored functor is fully-faithful, while the non-$\mathbb{A}^1$-invariant one is not in general.
\end{abstract}
\tableofcontents

\addtocontents{toc}{\protect\setcounter{tocdepth}{1}}

\section{Introduction}
In the thesis, we want to compare motivic homotopy theories and and noncommutative motives in the sense of \cite{BGT}. To be precise, we compare four motivic categories\begin{itemize}
\item non-$\mathbb{A}^1$-invariant motivic homotopy category $\mathcal{MS}$ of Annala, Iwasa and Hoyois \cite{AI}, \cite{AHI};
\item $\mathbb{A}^1$-invariant motivic homotopy category $\mathcal{SH}$ of Morel and Voevodsky;
\item category of localizing motives $\on{Mot}_{\on{loc}}$ of Blumberg, Gepner and Tabuada \cite{BGT};
\item category of $\mathbb{A}^1$-invariant localizing motives $\on{Mot}_{\on{loc}}^{\mathbb{A}^1}$
\end{itemize}

The comparisons between the first two, and the last two categories are given by localizations with respect to all projections $X\times\mathbb{A}^1\rightarrow X$. So we focus on the comparison between the first and the third categories, and the $\mathbb{A}^1$-invariant analogue between the second and the fourths. There is a natural candidate for the comparison functor between $\mathcal{SH}$ and $\on{Mot}_{\on{loc}}^{\mathbb{A}^1}$ (resp. between $\mathcal{MS}$ and $\on{Mot}_{\on{loc}}$), at least when we restrict the motivic homotopy category to the category of smooth schemes of finite presentation, that is,
\[\begin{aligned}
(Sm^{\on{fp}})^{op}&\rightarrow\on{Mot}_{\on{loc}}^{\mathbb{A}^1}\\
X&\mapsto\on{L}_{\mathbb{A}^1}\mathcal{U}(\on{perf}_X).
\end{aligned}\]
resp.\[\begin{aligned}
(Sm^{\on{fp}})^{op}&\rightarrow\on{Mot}_{\on{loc}}\\
X&\mapsto\mathcal{U}(\on{perf}_X).
\end{aligned}\]
The functor $\mathcal{U}$ above is the universal localizing invariant constructed in \cite{BGT}. This functoriality suggests us extending the domain to some kind of dual category of $\mathcal{SH}$ (resp. of $\mathcal{MS}$). It turns out that the desired notion of duality here is the dual object in the symmetric monoidal category of stable presentable categories $\PrL_{st}$, which we will denote by $(-)^\vee$. A first example is that, for a small category $\mathcal{C}$, the dual of the spectra-valued presheaf category is simply $\mathcal{P}sh(\mathcal{C}^{op})$, the presheaf category on $\mathcal{C}^{op}$. Such dualizibility is largely studied and applied in the work of Clausen-Scholze \cite{CS}, Efimov \cite{Efi}, etc.

To extend the natural candidate above to a functor \[\mathcal{SH}^\vee\rightarrow\on{Mot}_{\on{loc}}^{\mathbb{A}^1},\]
resp. \[\mathcal{MS}^\vee\rightarrow\on{Mot}_{\on{loc}},\]
we need a universal property that characterizes functors out of $\mathcal{SH}^\vee$ (resp. $\mathcal{MS}^\vee$). It turns out that just as $\mathcal{SH}$ (resp. $\mathcal{MS}$) is defined from presheaf category by first localizing with respect certain descent conditions, and then inverting the Tate circle $\mathbb{P}^1$, the dual category $\mathcal{SH}^\vee$ (resp. $\mathcal{MS}^\vee$) admits a dual description:
\begin{prop}
$\mathcal{SH}^\vee$ (resp. $\mathcal{MS}^\vee$) is equivalent to the category constructed from the spectra-valued presheaf category $\mathcal{P}sh((Sm^{\on{fp}})^{op})$, by first localizing it with respect to certain codescent conditions, and then formally inverting a dual version of $\mathbb{P}^1$.
\end{prop}
To deduce the proposition above, we computed the dual category for general formal inversion $\mathcal{C}[X^{-1}]$ for a stable, presentably symmetric monoidal, compactly generated category $\mathcal{C}$ and $X$ a compact object in $\mathcal{C}$:
\begin{prop}[Proposition \ref{formal.inversion.prop}]
Let $\mathcal{C}$ be a stable, compactly generated, presentably symmetric monoidal category. Let $X\in\mathcal{C}$ be a compact object. Then we have a canonical equivalence of categories
\[\mathcal{C}^\vee[(X^{dual})^{-1}]\simeq \mathcal{C}[X^{-1}]^\vee.\]
Moreover, the symmetric monoidal structures are compatible.
\end{prop}

Such characterization of $\mathcal{SH}^\vee$ (resp. $\mathcal{MS}^\vee$) as formal inversion of category of cosheaves is equipped with a natural universal property of the desired form (Proposition \ref{cons}):
\begin{prop}
The composition $\on{L}_{(\mathbb{P}^1)^{dual}}\circ\on{L}_{\on{coNis},\on{co}\mathbb{A}^1}\circ j$ induces, for each presentable, stable category $\mathcal{D}$, a fully faithful functor\[
\on{Fun}^L(\mathcal{SH}(S)^\vee,\mathcal{D})\rightarrow\on{Fun}(Sm^{\on{fp}}(S)^{op},\mathcal{D})
\]
with the essential image spanned by $\mathbb{A}^1$-invariant functors $F$ that preserve terminal object, send Nisnevich squares to pushouts, and send the fiber of \[fib(F(\mathbb{P}^1)\xrightarrow{F(\infty)}F(S))\]
to an invertible object.

Moreover, since the symmetric monoidal structures are compatible, the claim remains valid for symmetric monoidal category $\mathcal{D}$ and symmetric monoidal functors.

In particular, we have\[\on{Fun}^{L,\otimes}(\mathcal{SH}(S)^\vee,\on{Mot}(S)_{\on{loc}}^{\mathbb{A}^1})\simeq\on{Fun}^{\otimes}_{\on{Nis},\mathbb{A}^1,\mathbb{P}^1}(Sm^{\on{fp}}(S)^{op},\on{Mot}(S)_{\on{loc}}^{\mathbb{A}^1})
\]
\end{prop}

We also have a non-$\mathbb{A}^1$-invariant version (Proposition \ref{cons`}):
\begin{prop}
The composition $\on{L}_{(\mathbb{P}^1)^{dual}}\circ\on{L}_{\on{coNis},\on{coebu}}\circ j$ induces, for each presentable, stable category $\mathcal{D}$, a fully faithful functor\[
\on{Fun}^L(\mathcal{MS}(S)^\vee,\mathcal{D})\rightarrow\on{Fun}(Sm^{\on{fp}}(S)^{op},\mathcal{D})
\]
with the essential image spanned by functors $F$ that preserve terminal object, send Nisnevich squares and elementary blowup squares to pushouts, and send the fiber of \[fib(F(\mathbb{P}^1)\xrightarrow{F(\infty)}F(S))\]
to an invertible object.

Moreover, since the symmetric monoidal structures are compatible, the claim remains valid for symmetric monoidal category $\mathcal{D}$ and symmetric monoidal functors.

In particular, we have\[\on{Fun}^{L,\otimes}(\mathcal{MS}(S)^\vee,\on{Mot}(S)_{\on{loc}})\simeq\on{Fun}^{\otimes}_{\on{Nis},\on{ebu},\mathbb{P}^1}(Sm^{\on{fp}}(S)^{op},\on{Mot}(S)_{\on{loc}})
\]
\end{prop}
We check that the natural functors $(Sm^{\on{fp}})^{op}\rightarrow\on{Mot}_{\on{loc}}^{(\mathbb{A}^1)}$ satisfy the above conditions, and hence define the comparison functors.

A consequence of our universal properties is that the comparison functors are automatically strongly monoidal left adjoint functors, such functors factor formally through the category of modules $\on{Mod}_{\Phi_\ast(1)}(\mathcal{SH}^\vee)$ (resp. $\on{Mod}_{\Psi_\ast(1)}(\mathcal{MS}^\vee)$), where we use $\Phi_\ast$ (resp. $\Psi_\ast$) to denote the corresponding right adjoint functor. We are interested in the fully-faithfulness of the factored functor\[\on{Mod}_{\Phi_\ast(1)}(\mathcal{SH}^\vee)\rightarrow\on{Mot}_{\on{loc}}^{\mathbb{A}^1},\]
resp.\[\on{Mod}_{\Psi_\ast(1)}(\mathcal{MS}^\vee)\rightarrow\on{Mot}_{\on{loc}}.\]

In fact, restricted on schemes $X$ with sufficiently good dualizibility properties, both functors are fully faithful. A key difference appears when we consider the rigid generation of domains: working over a field $k$ that admits resolution of singularity, $\mathcal{SH}(k)^\vee$ is rigidly generated in the sense that it has a set of dualizable, compact generators. As a consequence, the restricted fully-faithfulness gives rise to a complete fully-faithfulness:
\begin{prop}[Proposition \ref{comparison.SH}]
Over a field $k$ and after inverting the exponential characteristic $e$, the functor\[
\widetilde{\Phi}:\on{Mod}_{\Phi_\ast(1)}(\mathcal{SH}(k)[\frac{1}{e}]^\vee)\rightarrow\on{Mot}(k)^{\mathbb{A}^1}_{\on{loc}}[\frac{1}{e}]
\]
is fully faithful on the level of mapping spectra.
\end{prop}
While in the non-$\mathbb{A}^1$-invariant case, $\mathcal{MS}(k)^\vee$ is not rigidly generated and we give an example of the failure of the fully-faithfulness in this setting. Two key ingredients here are the uncountability of mapping spectra between certain (compact) objects in $\on{Mot}_{\on{loc}}$ due to a theorem of Efimov (\cite[Theorem 9.1]{Efi25}, see also Section \ref{comparison.MS}), and the countability of mapping spectra between compact objects in general in $\mathcal{MS}(k)$:
\begin{prop}[Proposition \ref{countability}]
Let $k$ be a countable field and $e$ be the exponential characteristic of $k$. Then the full subcategory of $\mathcal{MS}(k)[\frac{1}{e}]^\vee$ spanned by compact objects is countable, in the sense that the $\pi_0$ of all the mapping spectra are countable.
\end{prop}

\subsection{Acknowledgements}
This is my master thesis at Sorbonne Université under the supervision of Vova Sosnilo. I want to express my great gratitude to him for proposing the topic of the thesis, sharing with me his ideas, answering my endless questions and giving feedback on the early version of this thesis. Moreover, I thank him sincerely for bringing me a wonderful summer thinking and discussing on these beautiful maths.

I am also very grateful to Marc Hoyois and Maxime Ramzi for pointing out several mistakes and for useful comments.

\addtocontents{toc}{\protect\setcounter{tocdepth}{3}}
\section{Preliminaries}\label{section:prelims}
Every category mentioned in the thesis is an $\infty$-category by default. 

\subsection{Categorical preliminaries}
\subsubsection{Formal inversion of an object}
In this section, we recall the techniques on formally inverting objects in a symmetric monoidal category. In the $\infty$-category setting, this is due to Robalo in \cite{Rob}. For a presentably symmetric monoidal category $\mathcal{C}$ and an object $X$ in $\mathcal{C}$, we want to formally inverting $X$ in a universal way. More precisely, we are hoping to find a symmetric monoidal category $\mathcal{C}[X^{-1}]$ and a strongly monoidal functor\[
\on{L}^{\mathcal{C}}_X: \mathcal{C}\rightarrow\mathcal{C}[X^{-1}],\]
such that $\on{L}^{\mathcal{C}}_X$ sends $X$ to an invertible object and is initial among all such functors.

The category of spectra $\on{Sp}$ is an example of this construction: it is the formal inversion of the category of pointed anima $\on{An}_\ast$ with respect to $(\mathbb{S}^1,1)$. Recall that besides the universal characterization, $\on{Sp}$ admits an explicit model\[
\on{Sp}=colim(\cdots\xrightarrow{-\otimes\mathbb{S}^1}\on{An}_\ast\xrightarrow{-\otimes\mathbb{S}^1}\on{An}_\ast\xrightarrow{-\otimes\mathbb{S}^1}\on{An}_\ast),\]
in which the colimit is computed in $\PrL$. Since this explicit model would be useful in several aspects, we also want to equip general formal inversions with models of the same kind. It turns out that with additional symmetric assumptions on the inverted object $X$, the model could be described in a similar way as that of $\on{Sp}$ above, which we will call a telescope model. While for more general $X$, the model is due to Annala and Iwasa in \cite{AI}, which is more subtle.

We begin from the case when $\mathcal{C}$ is a small symmetric monoidal category. That is the same data as an object in the category $\on{CAlg(\mathsf{Cat})}$, where $\mathsf{Cat}$ is equipped with the Cartesian monoidal structure. Recall that we have an equivalence of categories (see \cite[Proposition 3.16]{Rob})\[
\on{CAlg}(\on{Mod}_\mathcal{C}(\mathsf{Cat}))\simeq \on{CAlg}(\mathsf{Cat})_{\mathcal{C}/}.\]
We denote the subcategory of $\on{CAlg}(\mathsf{Cat})_{\mathcal{C}/}$ spanned by functors sending $X$ to an invertible object by $\on{CAlg}(\mathsf{Cat})_{\mathcal{C}/}^X$. This subcategory can be understood as a category of local objects. In fact, recall the adjunction\[
\on{Free}:\on{Mod}_\mathcal{C}(\mathsf{Cat})\rightleftarrows \on{CAlg}(\mathsf{Cat})_{\mathcal{C}/}:\on{U}.
\]
Let $\mathcal{S}_X$ be the collection of morphisms in $\on{CAlg}(\mathsf{Cat})_{\mathcal{C}/}$ consisting of a single map \[\on{Free}(X\otimes -):\on{Free}(\mathcal{C})\rightarrow\on{Free}(\mathcal{C}).\] We have
\begin{lm}(\cite[Proposition 4.1]{Rob})
The subcategory of $\on{CAlg}(\mathsf{Cat})_{\mathcal{C}/}$ spanned by the $\mathcal{S}_X$-local objects, in the sense of \cite[Definition 5.5.4.1]{HTT}, is equivalent to $\on{CAlg}(\mathsf{Cat})_{\mathcal{C}/}^X$. The inclusion\[
\on{CAlg}(\mathsf{Cat})_{\mathcal{C}/}^X\hookrightarrow \on{CAlg}(\mathsf{Cat})_{\mathcal{C}/}\]
admits a left adjoint (Bousfield localization with respect to $\mathcal{S}_X$), which we denote also by $(-)[X^{-1}]$.
\end{lm}
We call the unit $\on{L}_X(\mathcal{D}): \mathcal{D}\rightarrow \mathcal{D}[X^{-1}]$ of the localization above the formal inversion of $X$ in $\mathcal{D}$, and it is equipped the desired universal property:
\begin{lm}(\cite[Proposition 4.2]{Rob})
$\on{L}_X(\mathcal{C})$ induces a fully-faithful map of categories of algebras\[
\on{CAlg}(\mathsf{Cat})_{\mathcal{C}[X^{-1}]/}\hookrightarrow\on{CAlg}(\mathsf{Cat})_{\mathcal{C}/}\]
with essential image exactly $\on{CAlg}(\mathsf{Cat})^X_{\mathcal{C}/}$.
\end{lm}
We then make slight modification to adapt the presentable setting.
\begin{defn}\label{def.formal.inversion}
Let $free(\ast)$ be the free small symmetric monoidal category ($=$free commutative algebra in $\mathsf{Cat}^\times$) generated by $\left\{\ast\right\}$. Let $\mathcal{C}$ be a presentably symmetric monoidal category and $X$ be an object in $\mathcal{C}$. The formal inversion of $X$ in $\mathcal{C}$ is the presentably symmetric monoidal category $\mathcal{C}[X^{-1}]$ defined by the pushout
\[\begin{tikzcd}
	{Psh(free(\ast))} && {Psh(free(\ast)[\ast^{-1}])} \\
	\\
	{\mathcal{C}} && {\mathcal{C}[X^{-1}]}
	\arrow["\on{Lan}(j\circ\on{L}_\ast)", from=1-1, to=1-3]
	\arrow["X", from=1-1, to=3-1]
	\arrow[from=1-3, to=3-3]
	\arrow[from=3-1, to=3-3]
\end{tikzcd}\]
in $\on{CAlg}(\PrL)$, where the left arrow is induced essentially by the universal property of the category of presheaves as a presentable symmetric monoidal category.
\end{defn}
This presentable version is equipped with the following universal property:
\begin{lm}(Proposition 4.10 [Rob])\label{lm.universal.property.formal.inversion}
The inclusion $\on{CAlg}(\PrL)_{\mathcal{C}[X^{-1}]/} \hookrightarrow \on{CAlg}(\PrL)_{\mathcal{C}/}$ is fully-faithful with essential image $\on{CAlg}(\PrL)^X_{\mathcal{C}/}$.
\end{lm}
In general, for a presentably symmetric monoidal category $\mathcal{C}$ and $X$ an object, there is a model of the formal inversion $\mathcal{C}[X^{-1}]$ as a $\mathcal{C}^\otimes$-module called c-spectra defined and studied in \cite{AI} , which we denote by $\on{Sp}_X(\mathcal{C})$. Before recalling their construction, we want to mention that, in the special case when $X$ is symmetric in the sense that the cyclic permutation of $X\otimes X\otimes X$ is homotopic to $\on{id}_{X\otimes X\otimes X}$, the c-spectrum model reduces to  the classic telescope model.
\begin{prop}(\cite[Proposition 4.21]{Rob}, \cite[Proposition 1.6.3]{AI})\label{tel}
Let $\mathcal{C}$ be a presentably symmetric monoidal category and $X$ be a symmetric object. We have a well-defined $\mathcal{C}$-module:\[
Tel_X(\mathcal{C}):=colim(\cdots\rightarrow\mathcal{C}\xrightarrow{X}\mathcal{C}\xrightarrow{X}\mathcal{C}),\]
where the colimit is computed in $\on{Mod}_\mathcal{C}(\PrL)$. Moreover,\[
\mathcal{C}[X^{-1}]\simeq Tel_X(\mathcal{C})\simeq\on{Sp}_X(\mathcal{C})\]
as $\mathcal{C}$-modules.
\end{prop}
Now we recall the construction of c-spectra, following closely from \cite[Section 1.3]{AI}.

Let $\mathcal{C}$ be a presentably symmetric monoidal category, and let $B\Sigma_{\mathbb{N}}$ be the free commutative monoid in Anima with a single generator $e$. We consider the lax symmetric monoidal functor
\[(-)^{\Sigma}:=\on{Fun}(B\Sigma_{\mathbb{N}},-).\]
Then $\mathcal{C}^{\Sigma}$ is a presentably symmetric monoidal category given by Day convolution. We consider the adjunctions
\begin{itemize}
\item The left adjoint functor $F: \mathcal{C}\rightarrow\mathcal{C}^{\Sigma}$ is induced by the left Kan extension along the morphism $\ast\rightarrow B\Sigma_{\mathbb{N}}$ of commutative monoids. We denote its right adjoint by $U$;
\item The left adjoint functor $s_+:\mathcal{C}^{\Sigma}\rightarrow\mathcal{C}^{\Sigma}$ is induced by the left Kan extension along the morphism $e:B\Sigma_{\mathbb{N}}\rightarrow B\Sigma_{\mathbb{N}}$ of $B\Sigma_{\mathbb{N}}$-modules. We denote its right adjoint by $s_-$.
\end{itemize}
\begin{defn}\label{lax.}
Let $S_c$ be the free commutative algebra in $\mathcal{C}^{\Sigma}$ generated by $s_+\circ F(c)$. We define\[
\on{Sp}^{lax}_c(\mathcal{C}):=\on{Mod}_{S_c}(\mathcal{C}^{\Sigma}),\]
and call it the category of lax c-spectra in $\mathcal{C}$
\end{defn}
We consider $S_c$-linear version of adjunctions above, with a slight abuse of notations
\begin{itemize}
\item The left adjoint functor $F: \mathcal{C}\rightarrow\on{Sp}^{lax}_c(\mathcal{C})$ is induced by tensoring the $F: \mathcal{C}\rightarrow\mathcal{C}^{\Sigma}$ above with $S_c$. We denote its right adjoint by $U$;
\item The left adjoint functor $s_+:\on{Sp}^{lax}_c(\mathcal{C})\rightarrow\on{Sp}^{lax}_c(\mathcal{C})$ is induced by the $s_+:\mathcal{C}^{\Sigma}\rightarrow\mathcal{C}^{\Sigma}$ above. We denote its right adjoint by $s_-$.
\end{itemize}
\begin{defn}(\cite[Lemma 1.3.6, Construction 1.3.7, Definition 1.3.8]{AI})
A c-spectrum in $\mathcal{C}$ is a lax c-spectrum $E$ in $\mathcal{C}$ such that the adjoint\[
\sigma_E^{\sharp}:E\rightarrow (s_-E)^c\]
of the map induced by multiplication by $s_+\circ F(c)$\[
\sigma_E:s_+(c\otimes E)\rightarrow E\]
is an equivalence.

We denote by $\on{Sp}_c(\mathcal{C})$ the full subcategory of $\on{Sp}^{lax}_c(\mathcal{C})$ spanned by c-spectra.
\end{defn}

This construction indeed gives a model of formal inversion. To be precise, there is a $\mathcal{C}$-linear equivalence $\mathcal{C}[c^{-1}]\simeq\on{Sp}_c(\mathcal{C})$. See \cite[Proposition 1.3.14]{AI}.

In the following parts of the thesis, we are mostly interested in the stable presentable setting. One of the consequences of these explicit models is that, we could see, if we input a stable, presentably symmetric monoidal category, the output formal inversion would automatically be stable with the desired universal property in the world of stable categories:
\begin{lm}\label{lm.stable.universal.property}
Let $\mathcal{C}$ be a stable, presentably symmetric monoidal category, $X$ be an object in $\mathcal{C}$. Then the formal inversion $\mathcal{C}[X^{-1}]$ defined in Definition \ref{def.formal.inversion} is again stable. Moreover, this induces a restriction on the fully faithful embedding in Lemma \ref{lm.universal.property.formal.inversion}, that is, we have a fully faithful functor\[\on{CAlg}(\PrL_{st})_{\mathcal{C}[X^{-1}]/} \hookrightarrow \on{CAlg}(\PrL_{st})_{\mathcal{C}/}\]
with essential image $\on{CAlg}(\PrL_{st})^X_{\mathcal{C}/}$.
\end{lm}
\begin{proof}
Since the symmetric monoidal structure on $\PrL_{st}$ is compatible with that on $\PrL$, it suffices to show that $\mathcal{C}[X^{-1}]$ lands in $\PrL_{st}$. We consider the model of c-spectra. The category of lax c-spectra $\on{Sp}_c^{lax}(\mathcal{C})$ is stable when $\mathcal{C}$ is. The claim then follows from \cite[Lemma 1.3.9]{AI} that $\mathcal{C}[X^{-1}]$ is an accessible localization of $\on{Sp}_c^{lax}(\mathcal{C})$ and that the functors $s_+$ and $-\otimes-$ are exact.
\end{proof}

\subsubsection{Dualizable objects}
Recall that an object $X$ in a symmetric monoidal category $\mathcal{C}$ is dualizable if there exists some $X^\vee\in\mathcal{C}$ and (co)evaluation morphisms\[
ev: X^\vee\otimes X\rightarrow 1_\mathcal{C} \quad coev: 1_\mathcal{C}\rightarrow X^\vee\otimes X\]
satisfy triangle identities. More precisely, that says the following compositions
$$
\begin{gathered}
X \xrightarrow{\text { coev } \otimes i d} X \otimes X^\vee \otimes X \xrightarrow{\text { id } \otimes \mathrm{ev}} X \\
X^\vee \xrightarrow{\text { id } \otimes \text { coev }} X^\vee \otimes X \otimes X^\vee \xrightarrow{\text { ev } \otimes \text { id }} X^\vee
\end{gathered}
$$
are homotopic to the identity.

If the ambient symmetric monoidal category $\mathcal{C}$ is closed, and $X$ is a dualizable object in $\mathcal{C}$, then the evaluation map uniquely determines these structure maps, hence witnesses the inner hom $\underline{\on{Hom}}(X, 1)$ as a dual object of $X$. And when the dual object exists, it is unique. Throughout the thesis, we mostly handle the dualizable objects in the subcategory $\PrL_{st}$ of $\PrL$ spanned by stable categories and exact functors (and inheriting Lurie tensor product). In particular, if $\mathcal{C}\in\PrL_{st}$ is dualizable, then its dual is $\on{Fun}^L(\mathcal{C},\on{Sp})$.

A theorem of Lurie gives a full characterization of dualizable objects in $\PrL_{st}$. Before stating the result, we recall the definition of compactly assembled categories.
\begin{defn}
Let $\mathcal{C}$ be a category. A morphism $f: X\rightarrow Y$ is strongly compact if for every filtered colimit $Z=colim Z_i$, there is a lift rendering the following diagram commutative (up to homotopy):
\[\begin{tikzcd}
	{colim\on{Map}_{\mathcal{C}}(Y,Z_i)} && {colim\on{Map}_{\mathcal{C}}(X,Z_i)} \\
	\\
	{\on{Map}_{\mathcal{C}}(Y,Z)} && {\on{Map}_{\mathcal{C}}(X,Z)}
	\arrow["{f^\ast}", from=1-1, to=1-3]
	\arrow[from=1-1, to=3-1]
	\arrow[from=1-3, to=3-3]
	\arrow[dashed, from=3-1, to=1-3]
	\arrow["{f^\ast}", from=3-1, to=3-3].
\end{tikzcd}\]
\end{defn}
\begin{defn}
Let $\mathcal{C}$ be a category and $X\in\mathcal{C}$. We say that $X$ is strongly compactly exhausted if it is can be written as a sequential colimit\[
X=colim (X_0\rightarrow X_1\rightarrow X_2\rightarrow\cdots)\]
where all the transition maps $X_i\rightarrow X_{i+1}$ are strongly compact morphisms.
\end{defn}
\begin{defn}
A presentable category $\mathcal{C}$ is said to be compactly assembled if it is generated under colimits by strongly compactly exhausted objects.
\end{defn}
\begin{prop}(\cite[Proposition D.7.3.1]{SAG})
Let $\mathcal{C}$ be a stable and presentable category, the following are equivalent 
\begin{itemize}
\item $\mathcal{C}$ is dualizable in $\PrL_{st}$;
\item $\mathcal{C}$ is compactly assembled;
\item $\mathcal{C}$ is a retract in $\PrL_{st}$ of some compactly generated category.
\end{itemize}
\end{prop}

The right notion of functors between compactly assembled categories is compactly assembled functors. In the stable setting, they could be characterized as 
\begin{itemize}
\item strongly left adjoint functors, that is, left adjoint functors such that the right adjoint admits further a right adjoint, or,
\item functors that preserve strongly compact morphisms.
\end{itemize}
In the stable compactly generated setting, the equivalent conditions above are further equivalent to
\begin{itemize}
\item compactness-preserving functors.
\end{itemize}
We denote by $\PrL_{ca}$ the (non-full) subcategory of $\PrL_{st}$ spanned by stable, compactly assembled categories and compactly assembled functors between them. This category in fact inherits a symmetric monoidal structure from $\PrL_{st}$, which is a little subtle since it is not a full subcategory. We refer to \cite[Section 2.9]{KNP} for a detailed treatment.

Denote by $\PrL_{cg}$ the full subcategory of $\PrL_{ca}$ spanned by compactly generated categories, and denote by $\mathsf{Cat}^{\on{perf}}$ the category spanned by stable, idempotent complete categories and exact functors. We have the following well-known result:
\begin{lm}
The Ind-completion functor\[
\on{Ind}:\mathsf{Cat}^{\on{perf}}\rightarrow\PrL_{st}\]
induces an equivalence of $\mathsf{Cat}^{\on{perf}}$ and $\PrL_{cg}$, with the inverse given by\[
\begin{aligned}
(-)^\omega:\PrL_{cg}&\rightarrow\mathsf{Cat}^{\on{perf}}\\
\mathcal{C}&\mapsto\mathcal{C}^\omega.
\end{aligned}\]
\end{lm}
\begin{proof}
It essentially follows from \cite[Lemma 5.3.2.9]{HTT}. We take $\kappa=\omega$ in the lemma and restrict both sides to stable categories.
\end{proof}
Note that a dualizable object $\mathcal{C}$ in $\PrL_{st}$ (thus an object in $\PrL_{ca}$) may not be dualizable as an object in $\PrL_{ca}$ since the structure maps $ev$ and $coev$ may not be compactly assembled maps. In the following of the thesis, we call a category $\mathcal{C}$ dualizable if it is a dualizable object in $\PrL_{st}$, and we call such a category $\mathcal{C}$ smooth (resp. proper) when $coev$ (resp. $ev$) happens to be compactly assembled. We call a category saturated when it is both smooth and proper.\\

\subsubsection{A compatibility result of dual operation and formal inversion}\label{compactibility}
Let $\mathcal{C}\in\on{CAlg}(\PrL_{ca})$ and $X\in\mathcal{C}^{\omega}$. The main goal of this section is to show that taking dual commutes with formal inversion. Note first that we have the following commutative diagram by the combination of the last lemma in the last section and \cite[Construction 2.9.21]{KNP}
\[\begin{tikzcd}
	{\mathsf{Cat}^{\on{perf}}} && {\PrL_{cg}} && {\PrL_{ca}} \\
	\\
	{\mathsf{Cat}^{\on{perf}}} && {\PrL_{cg}} && {\PrL_{ca}}
	\arrow["{\on{Ind}(-)}", from=1-1, to=1-3]
	\arrow["\sim"', from=1-1, to=1-3]
	\arrow["{(-)^{op}}", from=1-1, to=3-1]
	\arrow[hook, from=1-3, to=1-5]
	\arrow["{(-)^\vee}", from=1-3, to=3-3]
	\arrow["\sim"', from=1-3, to=3-3]
	\arrow["{(-)^\vee}", from=1-5, to=3-5]
	\arrow["\sim"', from=1-5, to=3-5]
	\arrow["{\on{Ind}(-)}", from=3-1, to=3-3]
	\arrow["\sim"', from=3-1, to=3-3]
	\arrow[hook, from=3-3, to=3-5],
\end{tikzcd}\]
where all functors are symmetric monoidal. For a compactly assembled category $\mathcal{C}$, the vertical functor $(-)^\vee$ sends it to $\mathcal{C}^\vee\simeq\on{Fun}^L(\mathcal{C},\on{Sp})$, while for a strongly left adjoint functor $L:\mathcal{C}\rightarrow\mathcal{D}$, \[L^\vee:\on{Fun}^L(\mathcal{C},\on{Sp})\rightarrow\on{Fun}^L(\mathcal{D},\on{Sp})\] is given by the precomposition with the right adjoint $R$ of $L$. Equivalently, it is the left adjoint of the functor given by precomposite with $L$. In particular, for $\mathcal{C}\in\PrL_{ca}$ and $X\in\mathcal{C}$ compact, the colimits-preserving functor\[
\on{Sp}\xrightarrow{X}\mathcal{C}\]
is strongly left adjoint, hence dualizes to a functor
\[\on{Sp}^\vee\rightarrow\mathcal{C}^\vee.\]
Since $\on{Sp}\simeq\on{Sp}^\vee$, we precomposite it with the functor above, and the composition\[
\on{Sp}\rightarrow\on{Sp}^\vee\rightarrow\mathcal{C}^\vee\]
is determined by a single element in $\mathcal{C}^\vee$, which we denote by $X^{dual}$.

The first observation is that the formal inversion of a compactly assembled (resp. compactly generated) category with respect to a compact object is still compactly assembled (resp. compactly generated). We can see this from the c-spectra model of formal inversion. In fact, it suffices to show that the localization functor $\on{L}$ preserves weakly compactly exhaustible objects (resp. compact objects). This is true since the right adjoint preserves filtered colimits, as proved in \cite[Lemma 1.5.2]{AI}.

Combining all the discussions so far, the question whether taking dual commutes with formal inversion makes sense. We give a positive answer:
\begin{prop}\label{formal.inversion.prop}
Let $\mathcal{C}$ be a stable, compactly assembled, presentably symmetric monoidal category. Let $X\in\mathcal{C}$ be a compact object. Then \[\mathcal{C}^\vee[(X^{dual})^{-1}]\simeq \mathcal{C}[X^{-1}]^\vee\]
holds in $\on{CAlg}(\mathcal{P}r^L_{ca})_{\mathcal{C}^\vee[(X^{dual})^{-1}]/}$. A fortiori, they are the same as presentably symmetric monoidal category.
\end{prop}
\begin{proof}
Note first that $\mathcal{C}[X^{-1}]^\vee$ is a $\mathcal{C}^\vee$ commutative algebra via the dual of the canonical monoidal functor $\on{L}_X^\mathcal{C}:\mathcal{C}\rightarrow\mathcal{C}[X^{-1}]$. Again, $\on{L}_X^\mathcal{C}$ is strongly left adjoint by \cite[Lemma 1.5.2]{AI}. This ring map factors through $\mathcal{C}^\vee[(X^{dual})^{-1}]$ by the next lemma and the universal property of formal inversion. On the other hand, since symmetric monoidal structure on $\PrL_{ca}$ is restricted from $\PrL_{st}$, Lemma \ref{lm.stable.universal.property} remains valid in compactly assembled world. The Yoneda lemma then applies:\[
\begin{aligned}
\on{CAlg}(\PrL_{ca})_{\mathcal{C}^\vee[(X^{dual})^{-1}]/}(\mathcal{C}^\vee[(X^{dual})^{-1}],\mathcal{D})&\simeq \on{CAlg}(\PrL_{ca})_{\mathcal{C}^\vee/}(\mathcal{C}^\vee,\mathcal{D})\\
&\simeq \on{CAlg}(\PrL_{ca})_{\mathcal{C}/}(\mathcal{C},\mathcal{D}^\vee)\\
&\simeq \on{CAlg}(\PrL_{ca})_{\mathcal{C}[X^{-1}]/}(\mathcal{C}[X^{-1}],\mathcal{D}^\vee)\\
&\simeq \on{CAlg}(\PrL_{ca})^X_{\mathcal{C}/}(\mathcal{C}[X^{-1}],\mathcal{D}^\vee)\\
&\simeq \on{CAlg}(\PrL_{ca})^{X^{dual}}_{\mathcal{C}^\vee/}(\mathcal{C}[X^{-1}]^\vee,\mathcal{D})\\
&\simeq \on{CAlg}(\PrL_{ca})_{\mathcal{C}^\vee[(X^{dual})^{-1}]/}(\mathcal{C}[X^{-1}]^\vee,\mathcal{D}),
\end{aligned}\]
where the third and the fifth isomorphisms follow again from the next lemma.
\end{proof}
\begin{lm}
Let  $\mathcal{C}\in\on{CAlg}(\PrL_{ca})$ with $X$ a compact object in $\mathcal{C}$. Let $\mathcal{E}\in\on{CAlg}(\PrL_{ca})_{\mathcal{C}/}$. Then $-\otimes X:\mathcal{E}\rightarrow\mathcal{E}$ dualizes to $-\otimes X^{dual}:\mathcal{E}^\vee\rightarrow\mathcal{E}^\vee.$
\end{lm}
\begin{proof}
Since $X$ is compact, $-\otimes X$ preserves compact morphisms, hence is strongly left adjoint.

For a presentable category $\mathcal{D}$, we denote by $-\boxtimes-$ the natural map from $\mathcal{D}\times\mathcal{D}$ to $\mathcal{D}\otimes\mathcal{D}$. We consider the commutative diagram
\[\begin{tikzcd}
	{\mathcal{E}^\vee} && {\mathcal{E}^\vee\otimes\mathcal{E}^\vee} && {\mathcal{E}^\vee} \\
	\\
	{\operatorname{Fun}^L(\mathcal{E},\operatorname{Sp})} && {\operatorname{Fun}^L(\mathcal{E}\otimes\mathcal{E},\operatorname{Sp})} && {\operatorname{Fun}^L(\mathcal{E},\operatorname{Sp})}
	\arrow["{\operatorname{id}\boxtimes X^{dual}}", from=1-1, to=1-3]
	\arrow["\simeq"', from=1-1, to=3-1]
	\arrow["{mult^\vee}", from=1-3, to=1-5]
	\arrow["\simeq"', from=1-3, to=3-3]
	\arrow["\simeq"', from=1-5, to=3-5]
	\arrow[shift left=2, from=3-1, to=3-3]
	\arrow["{(\operatorname{id}\boxtimes X)^\ast}", shift left=2, from=3-3, to=3-1]
	\arrow[shift left=2, from=3-3, to=3-5]
	\arrow["{mult^\ast}", shift left=2, from=3-5, to=3-3],
\end{tikzcd}\]
where $mult^\vee$ is the monoidal structure map on $\mathcal{E}^\vee$ induced by that on $\mathcal{E}$. Alternatively, that is the left adjoint of the functor\[
-\circ mult: \mathcal{E}^\vee\simeq\on{Fun}^L(\mathcal{E},\on{Sp})\rightarrow\on{Fun}^L(\mathcal{E}\otimes\mathcal{E},\on{Sp})\simeq(\mathcal{E}\otimes\mathcal{E})^\vee\simeq\mathcal{E}^\vee\otimes\mathcal{E}^\vee.\] It is then immediate that the upper functors composite to $-\otimes X^{dual}$.

On the other hand, we claim that $-\otimes X^{dual}$ is the dual of $-\otimes X$. In fact, it follows from the definition of $X^{dual}$ and $mult^\vee$ that they are precisely the left adjoints of the precompositions with $\on{id}\boxtimes X$ and with the monoidal structure map $mult$ respectively. But\[
(\on{id}\boxtimes X)^\ast\circ mult^\ast\simeq (-\otimes X)^\ast\]
The left adjoint of that is the dual of $-\otimes X$ by definition of the functor $(-)^\vee$.
\end{proof}
\begin{rmk}\label{dual.comp}
We also mention that for a compactly generated, presentable, stable category $\mathcal{C}$ and $X$ compact in $\mathcal{C}$, the induced $X^{dual}\in\mathcal{C}^\vee$ identifies with $j(X)$ via the equivalence\[
\begin{aligned}
\mathcal{C}^\vee&\simeq\on{Ind}((\mathcal{C}^\omega)^{op})\\
X^{dual}&\leftrightsquigarrow j(X)
\end{aligned}\]
This is straightforward, as $X^{dual}$ is the same data as a left adjoint functor\[\begin{aligned}
\mathcal{C}&\rightarrow\on{Sp}\\
Y&\mapsto\on{Map}(X,Y).
\end{aligned}\]
But this is the spectral Yoneda embedding on $(\mathcal{C}^\omega)^{op}$.
\end{rmk}

\subsubsection{Factorization of adjunction pair through category of modules}\label{fac}
This section follows from \cite{EK}. Let $\gamma^\ast: \mathcal{C} \rightleftarrows \mathcal{D}: \gamma_\ast$ be an adjunction between symmetric monoidal categories and we assume the left adjoint functor $\gamma^\ast$ to be strongly symmetric monoidal, then the adjunction pair factors through the following commutative diagram,
\[\begin{tikzcd}
	{\mathcal{C}} && {\mathcal{D}} \\
	& {\on{LMod}_{\gamma_\ast\gamma^\ast}(\mathcal{C})} \\
	& {\on{LMod}_{\gamma_\ast\gamma^\ast(1)\otimes-}(\mathcal{C})}
	\arrow["{\gamma^\ast}", from=1-1, to=1-3]
	\arrow[from=1-1, to=2-2]
	\arrow[from=1-1, to=3-2]
	\arrow["{{(\gamma^\ast)^{enh}}}", from=2-2, to=1-3]
	\arrow[from=3-2, to=1-3]
	\arrow["\alpha"', from=3-2, to=2-2],
\end{tikzcd}\]
in which the notation $\on{LMod}$ means the category of modules over a monad. Each of the arrows above admits a right adjoint, and if we replace all arrows by their right adjoints, the diagram remains commutative. In the following part of this section, we explain the diagram and give conditions for $(\gamma^\ast)^{enh}$, $\alpha$ to be equivalences.

The factorization through $\on{LMod}_{\gamma_\ast\gamma^\ast}(\mathcal{C})$ is true even without any symmetric monoidal structure. The idea is that the universal property of the monad $\gamma_\ast\gamma^\ast$ provides a factorization\[
\mathcal{D}\xrightarrow{(\gamma_\ast)^{enh}}\on{LMod}_{\gamma_\ast\gamma^\ast}(\mathcal{C})\xrightarrow{\on{U}}\mathcal{C}.\]
Informally speaking, $(\gamma_\ast)^{enh}$ sends $d\in\mathcal{D}$ to $\gamma_\ast(d)$, and the monad $\gamma_\ast\gamma^\ast$ acts on $\gamma_\ast(d)$ via the counit of the adjunction. We refer to \cite[proposition 4.7.3.3]{HA} for a detailed discussion. Moreover, the two maps involved above are both right adjoints. And we denote by $(\gamma^\ast)^{enh}$ the left adjoint of the first functor. The Lurie-Barr-Beck theorem gives an necessary and sufficient condition for $(\gamma^\ast)^{enh}$ to be an equivalence:
\begin{lm}(\cite[Theorem 4.7.3.5]{HA})
Let $\gamma^\ast: \mathcal{C} \rightleftarrows \mathcal{D}: \gamma_\ast$ be an adjunction. The following are equivalent:
\begin{enumerate}
\item The functors $(\gamma_\ast)^{enh}$ and $(\gamma^\ast)^{enh}$ are mutually inverse.
\item The functor $\gamma_\ast$ satisfies the following conditions:
\begin{itemize}
\item $(\gamma_\ast)^{enh}$ is conservative;
\item For all $\gamma_\ast$-split simplicial object $X$ in $\mathcal{D}$, $X$ admits a colimit in $\mathcal{D}$. Furthermore, $\gamma_\ast$ preserves and reflects $X$-indexed colimits.
\end{itemize}
\end{enumerate}
\end{lm}
The factorization through $\on{LMod}_{\gamma_\ast\gamma^\ast(1)\otimes-}(\mathcal{C})$ is induced by a map of monads \[a: \gamma_\ast\gamma^\ast(1)\otimes-\rightarrow\gamma_\ast\gamma^\ast.\]More precisely, $a$ is the composite of the following maps of monads:\[
\begin{aligned}
\gamma_\ast\gamma^\ast(1)\otimes- &\simeq \gamma_\ast\gamma^\ast(1)\otimes\on{id}(-)\\
&\xrightarrow{\on{id}\otimes\eta}\gamma_\ast\gamma^\ast(1)\otimes\gamma_\ast\gamma^\ast(-)\\
&\xrightarrow{\mu}\gamma_\ast(\gamma^\ast(1)\otimes\gamma^\ast(-))\\
&\simeq\gamma_\ast\gamma^\ast,
\end{aligned}\] 
where $\mu$ is given by the lax monoidal structure on $\gamma_\ast$. This $a$ induces a functor $\alpha_\ast: \on{LMod}_{\gamma_\ast\gamma^\ast}(\mathcal{C})\rightarrow\on{LMod}_{\gamma_\ast\gamma^\ast(1)\otimes-}(\mathcal{C})$, which is a right adjoint by the adjoint functor theorem. The left adjoint is the desired $\alpha$, which fits in the left side commutative triangle since all the right adjoints of arrows in the triangle are forgetful functors. We will now explain that under some rigid generation conditions, $\alpha$ becomes an equivalence.
\begin{defn}
Let $\mathcal{C}$ be a stable, compactly generated, presentably monoidal category. We say that $\mathcal{C}$ is rigidly generated if $\mathcal{C}$ is generated under filtered colimits by (a small set of) compact, dualizable objects.
\end{defn}
\begin{lm}\label{same.modules}
Let $\gamma^\ast: \mathcal{C} \rightleftarrows \mathcal{D}: \gamma_\ast$ be an adjunction of symmetric monoidal categories. Suppose that $\gamma^\ast$ is monoidal, $\mathcal{C}$ is rigidly generated, $\mathcal{D}$ has compact unit, and colimits distribute over tensor product in $\mathcal{C}$ and $\mathcal{D}$, then $\alpha$ is fully faithful. Moreover, if $\gamma_\ast$ is conservative, then $\alpha$ is an equivalence.
\end{lm}
\begin{proof}
We give a sketch of proof. The idea follows \cite[Appendix A]{BHS}.\\
It suffices to show the projection formula holds. That is, for all $X$ and $Y$,\[
p: \gamma_\ast Y\otimes X\xrightarrow{\on{id}\otimes\eta}\gamma_\ast Y\otimes\gamma_\ast\gamma^\ast X\xrightarrow{\mu}\gamma_\ast(Y\otimes\gamma^\ast X)
\]
is an isomorphism. This is true for dualizable $X$, since by the Yoneda lemma, we have\[
\begin{aligned}
\mathcal{C}(Z,\gamma_\ast Y\otimes X)&\simeq \mathcal{C}(Z\otimes X^\vee,\gamma_\ast Y)\\
&\simeq \mathcal{D}(\gamma^\ast Z\otimes\gamma^\ast (X^\vee), Y)\\
&\simeq \mathcal{D}(\gamma^\ast Z\otimes(\gamma^\ast X)^\vee, Y)\\
&\simeq \mathcal{D}(\gamma^\ast Z, Y\otimes\gamma^\ast X)\\
&\simeq \mathcal{C}(Z, \gamma_\ast(Y\otimes\gamma^\ast X)).
\end{aligned}\]
Moreover, by a diagram chasing, we see that the above isomorphism is induced precisely by $p$. To show the general case, since colimits in $\mathcal{C}$ and $\mathcal{D}$ commute with tensor product and $\mathcal{C}$ is rigidly generated, it reduces to show that $\gamma_\ast$ preserves colimits. But this is true since $\gamma^\ast$ sends compact dualizable objects (which form a family of generators) to compact objects by the compactness of unit in $\mathcal{D}$.
\end{proof}
Finally, we have a general equivalence between categories of modules over an algebra and that over a monad of a certain form:
\begin{lm}
Let $\mathcal{C}$ be a symmetric monoidal category and $A$ be an $\mathbb{E}_\infty$-algebra in $\mathcal{C}$. If we denote by $\on{Mod}_A(\mathcal{C})$ the category of modules over $A$, then we have an equivalence
\[
\beta:\on{Mod}_A(\mathcal{C})\rightarrow\on{Mod}_{A\otimes-}(\mathcal{C}).\]
\end{lm}
\begin{proof}
Apply \cite[Theorem 4.7.3.5]{HA} to the forgetful functor $\on{Mod}_A(\mathcal{C})\rightarrow\mathcal{C}$.\end{proof}

\subsection{Category of motives}
\subsubsection{$\mathcal{SH}$}\label{SH}
In this section, we recall Morel-Voevodsky's construction of stable motivic homotopy category $\mathcal{SH}$. We emphasis especially the symmetric monoidal structure on $\mathcal{SH}$ and the compact generation of $\mathcal{SH}$. We work with a slightly more general setting as in \cite[Appendix C]{Hoy}.

We begin by recalling some topos theory. Let $\mathcal{C}$ be a category. A quasi-topology $\tau$ on $\mathcal{C}$ assigns, to each $X\in\mathcal{C}$, a collection of sieves $\tau(X)$ on $X$, such that, for every $f: Y\rightarrow X$, $f^\ast\tau(X)\subset\tau(Y).$ A presheaf $F$ on $\mathcal{C}$ is a $\tau$-sheaf if it is local with respect to all $R\hookrightarrow X$, where $R\in\tau(X)$. Let $\overline{\tau}$ be the coarsest Grothendieck topology containing $\tau$, \cite[Corollary C.2]{Hoy} shows that the category of $\tau$-sheaves is the same as the category of $\overline{\tau}$-sheaves, where the latter is in the classical sense. In the following, we denote the equivalent category by $Shv_\tau(\mathcal{C}).$

Let $S$ be a scheme. We denote by $Sm(S)$ the category of smooth $S$-schemes and by $Sm^{\on{fp}}(S)$ the subcategory of smooth $S$-schemes of finite presentation. We consider the Nisnevich topology on $Sm(S)$ (resp. $Sm^{\on{fp}}(S)$) as the coarsest Grothendieck topology containing the Nisnevich quasi-topology: $\operatorname{Nis}=\operatorname{Zar} \cup \operatorname{Nis}_{\mathrm{qc}}$ (resp. $\operatorname{Nis}=\operatorname{Nis}^{\on{fp}}_{\mathrm{qc}}$), where $\operatorname{Zar}$ means the usual Zariski covering sieves and $\operatorname{Nis}_{\mathrm{qc}}$ (resp. $\operatorname{Nis}^{\on{fp}}_{\mathrm{qc}}$) consists of sieves:
\begin{itemize}
\item the empty sieve on $\emptyset$;
\item for every Nisnevich square (resp. Nisnevich square of finite presentation), that is, a Cartesian square of schemes
\[\begin{tikzcd}
	W && V \\
	\\
	U && X
	\arrow[hook, from=1-1, to=1-3]
	\arrow[from=1-1, to=3-1]
	\arrow["p", from=1-3, to=3-3]
	\arrow["j", hook, from=3-1, to=3-3],
\end{tikzcd}\]
where $j$ is an open immersion, $p$ is étale, and there exists a closed immersion $Z \hookrightarrow X$ complement to $U$ such that $p$ induces an isomorphism $V \times_X Z \simeq Z$ (resp. moreover, $j$ and $p$ are finite presented), the sieve generated by $\left\{j, p\right\}$.
\end{itemize}
We say that a presheaf $F$ on $Sm(S)$ (resp. $Sm^{\on{fp}}(S)$) satisfies Nisnevich excision if and only if:
\begin{itemize}
\item $F(\emptyset)\simeq\ast$;
\item for every Nisnevich square $Q$ (resp. Nisnevich square of finite presentation), $F(Q)$ is Cartesian. 
\end{itemize}
\begin{lm}(Proposition C.5.(2) [Hoy])\label{Sheaf.Theory.of.Nis.}
Let $S$ be an arbitrary base scheme, then:
\begin{itemize}
\item A presheaf on $Sm(S)$ is a Nisnevich sheaf if and only if it satisfies Zariski descent and Nisnevich excision.
\end{itemize}
If $S$ is moreover qcqs, then:
\begin{itemize}
\item A presheaf $F$ on $Sm(S)$ is a Nisnevich sheaf if and only if it is a right Kan extension of a Nisnevich sheaf on $Sm^{\on{fp}}(S)$. In particular, $Shv_{\on{Nis}}(Sm(S))\simeq Shv_{\on{Nis}}(Sm^{\on{fp}}(S))$.
\item A presheaf on $Sm^{\on{fp}}(S)$ is a Nisnevich sheaf if and only if it satisfies Nisnevich excision.
\end{itemize}
\end{lm}

Now we can define the unstable motivic homotopy category $\mathcal{H}(S)$. We say that a presheaf $F$ on $Sm(S)$ is $\mathbb{A}^1$-invariant if, for every $X \in Sm(S)$, the projection $\mathbb{A}^1 \times X \rightarrow X$ induces an equivalence $F(X) \simeq F\left(\mathbb{A}^1 \times X\right)$. We denote by $\mathcal{H}(S)$ the full subcategory of $Psh(Sm(S))$ spanned by $\mathbb{A}^1$-invariant Nisnevich sheaves. When $S$ is qcqs, it follows the lemma above that $\mathcal{H}(S)$ is presentable and the inclusion $\mathcal{H}(S)\hookrightarrow Psh(Sm(S))$ admits a left adjoint $\on{L}_{\mathbb{A}^1,\on{Nis}}$ (this is still true for a general base scheme $S$, see \cite[Proposition C.5.(1)]{Hoy}). From now on, we fix $S$ to be qcqs.

Note that the subcategories of $Psh(Sm^{\on{fp}}(S))$ spanned by $\mathbb{A}^1$-invariant presheaves and presheaves satisfying Nisnevich excision are closed under filtered colimits since they are detected by finite limits. Hence $\mathcal{H}(S)$ is closed under filtered colimits as a subcategory of $Psh(Sm^{\on{fp}}(S))$. In particular, every $X\in Sm^{\on{fp}}(S)$ is compact in $\mathcal{H}(S)$. Every presheaf is a colimit of representables, thus a filtered colimits of finite colimits of representables. Since the right adjoint $\mathcal{H}(S)\hookrightarrow Psh(Sm(S))$ preserves filtered colimits, $\on{L}_{\mathbb{A}^1,\on{Nis}}$ preserves compactness. As a consequence, every object $F\in\mathcal{H}(S)$ is of the form\[
F\simeq fil.colim.(\on{L}_{\mathbb{A}^1,\on{Nis}}(fin.colim. (X_i))),\]
where $X_i\in Sm^{\on{fp}}(S)$ and the finite colimits are computed in the presheaf category. This witnesses the compact generation of $\mathcal{H}(S)$.

Since $\mathcal{H}(S)$ is Cartesian, we equip it with a symmetric monoidal structure in which the tensor product is given by Cartesian product. It follows from the general facts that the pointed category $\mathcal{H}_\ast(S)$ admits a unique symmetric monoidal structure $\wedge$ which is compatible with colimits and such that the functor $(-)_+: \mathcal{H}(S)\rightarrow \mathcal{H}_\ast(S)$ is symmetric monoidal.

We then define $\mathcal{SH}(S)$ to be the formal inversion of $\mathcal{H}_\ast(S)$ with respect to $(\mathbb{P}^1,\infty)$, which is automatically symmetric monoidal. Since $\mathbb{P}^1$ is symmetric in motivic space (\cite[Theorem 4.3, Lemma 4.4]{Voe}), the telescope model applies. This gives the compact generation of $\mathcal{SH}(S)$.
\begin{lm}(\cite[Proposition C.12]{Hoy})
Let $\Sigma^\infty_+ : \mathcal{H}(S)\rightarrow \mathcal{SH}(S)$ be the composition of adding a point and the formal inversion functor. Assuming the base scheme $S$ is qcqs, then the followings are true:
\begin{enumerate}
\item $\mathcal{SH}(S)$ is generated under colimits by objects of the form $\Sigma^{-n}_{\mathbb{P}^1}\Sigma^\infty_+X$ for $X\in Sm^{\on{fp}}(S)$ and $n\geq 0$.
\item $\Sigma^\infty_+X$ is compact for $X\in Sm^{\on{fp}}(S)$.
\item $\mathcal{SH}(S)$ is compactly generated.
\end{enumerate}
\end{lm}
\begin{proof}
Note first that $\mathcal{H}_\ast(S)$ is compactly generated by $X_+$ with $X\in Sm^{\on{fp}}(S)$.
\begin{enumerate}
\item Recall the telescope model: \[
\begin{aligned}
\mathcal{SH}(S)&\simeq colim (\cdots\mathcal{H}_\ast(S)\xrightarrow{\Sigma_{\mathbb{P}^1}}\mathcal{H}_\ast(S)\xrightarrow{\Sigma_{\mathbb{P}^1}}\mathcal{H}_\ast(S))\\
&\simeq lim (\cdots\mathcal{H}_\ast(S)\xrightarrow{\Omega_{\mathbb{P}^1}}\mathcal{H}_\ast(S)\xrightarrow{\Omega_{\mathbb{P}^1}}\mathcal{H}_\ast(S)).
\end{aligned}\]
For each $E\in\mathcal{SH}(S)$, in its second presentation, if we denote the $n$-th component by $E_n$, by \cite[lemma 6.3.3.6]{HA}, we have\[
E\simeq colim_\mathbb{N}\Sigma^{-n}_{\mathbb{P}^1}\Sigma^\infty_{\mathbb{P}^1}E_n,
\]
and each $E_n\in\mathcal{H}_\ast(S)$, thus is of the form $colim X_+$. By commuting these colimits, we get the desired form.
\item Since such $X_+$ is compact in $\mathcal{H}_\ast(S)$, it suffices to show that $\Sigma^\infty$ preserves compact objects, alternatively, the right adjoint $\Omega^\infty$ preserves filtered colimits. Again by the second telescope construction, it suffices to show $\Omega_{\mathbb{P}^1}$ preserves filtered colimits. This is true since $(\mathbb{P}^1,\infty)$ is compact.
\item Combine (1) and (2).
\end{enumerate}
\end{proof}

\subsubsection{$\mathcal{MS}$}
We recall the stable, non-$\mathbb{A}^1$-invariant motivic homotopy category $\mathcal{MS}$ constructed by Annala-Iwasa. This section follows from \cite{AHI} and \cite{AI}.

In this section, and in the following sections, when the context is non-$\mathbb{A}^1$-invariant motivic homotopy theory, by schemes we mean derived schemes by default, and we denote the $\infty$-category of (derived) smooth $S$-schemes by $Sm(S)$. We begin by allowing arbitrary base scheme $S$, and then restrict ourselves to only qcqs base schemes to ensure compact generation of motivic homotopy categories and compactness of $\mathbb{P}^1$. We first recall the notion of smooth (resp. elementary) blowup excision. For the notion of blowup of a derived scheme at a quasi-smooth closed subscheme, we refer to \cite{KR}.
\begin{defn}
Let $\mathcal{C}$ be a category and $F: Sm(S)^{op}\rightarrow \mathcal{C}$ be a $\mathcal{C}$-valued presheaf.
\begin{enumerate}
\item We say that $F$ satisfies smooth blowup excision if $F(\emptyset)$ is a final object of $\mathcal{C}$ and for every closed immersion $i: Z \hookrightarrow X$ in $Sm(S)$, $F$ sends the blowup square
\[\begin{tikzcd}
	E && {Bl_ZX} \\
	\\
	Z && X
	\arrow[hook, from=1-1, to=1-3]
	\arrow[from=1-1, to=3-1]
	\arrow[from=1-3, to=3-3]
	\arrow["i", hook, from=3-1, to=3-3]
\end{tikzcd}\]
to a cartesian square.
\item A closed immersion $i: Z \hookrightarrow X$ is called elementary if, Zariski-locally on $X$, it is the zero section of $\mathbb{A}_Z^n \sqcup Y$ for some $n \geq 0$ and some $Y$. We say that $F$ satisfies elementary blowup excision if (1) holds whenever $i$ is elementary.
\end{enumerate}
We denote by $Psh_{\on{sbu}}(Sm(S))\subset Psh_{\on{ebu}}(Sm(S))$ the full subcategories of $Psh(Sm(S))$ of presheaves satisfying the corresponding excision, and by $\on{L}_{\on{sbu}}$ and $\on{L}_{\on{ebu}}$ the corresponding localization functors, which preserve finite products.
\end{defn}
For Nisnevich sheaves of spectra, there is no difference between elementary and smooth blowup excision:
\begin{lm}(\cite[Proposition 2.2]{AHI})
Suppose that $\mathcal{C}$ is stable and that $F: Sm(S)^{op} \rightarrow \mathcal{C}$ satisfies Nisnevich descent and elementary blowup excision. Then F satisfies smooth blowup excision.
\end{lm}
We denote, for each scheme $S$, by $\mathcal{P}sh_{\on{Nis},\on{ebu}}(Sm(S))$ the subcategory of the category of spectra-valued presheaves $\mathcal{P}sh(Sm(S))$ spanned by those presheaves satisfying Nisnevich descent and elementary blowup excision. Similar to the $\mathbb{A}^1$-invariant case, this subcategory inherits a symmetric monoidal structure due to formal reasons.

Then we define the stable non-$\mathbb{A}^1$-invariant motivic homotopy category $\mathcal{MS}(S)$ as a formal inversion with respect to $\mathbb{P}^1$. That is,
\[\mathcal{MS}(S):=\mathcal{P}sh_{\on{Nis},\on{ebu}}(Sm(S))[(\mathbb{P}^1)^{-1}].\]
\begin{rmk}
In the construction of $\mathcal{MS}$, we begin from the category of spectra-valued presheaves instead of that of anima-valued presheaves. One of the reasons is that after abandoning the $\mathbb{A}^1$-invariance, $\mathbb{P}^1$ is no longer a tensor product of $\mathbb{S}^1$ and $\mathbb{G}_m$, and we need to stabilize the category by hand.
\end{rmk}
\begin{rmk}
The fact that $(\mathbb{P}^1,\infty)$ is symmetric in the sense of Proposition \ref{tel} also relies on the $\mathbb{A}^1$-invariance. Therefore the category $\mathcal{MS}$ does not admit the classical telescope presentation in general. This is one of the reasons why Annala-Iwasa introduced the notion of c-spectra.
\end{rmk}
Over a qcqs base $S$, $\mathcal{MS}(S)$ is generated under colimits by compact objects of the form $\Sigma^{-n}\Sigma^{\infty}_+X$. This follows from a similar argument of the compact generation of $\mathcal{SH}(S)$.

\subsubsection{$\on{Mot}_{\on{loc}}$}\label{E-linear.NC}
In this section, we recall the category of $\mathcal{E}$-linear localizing motives and the corepresentability of K-theory. This section follows mainly from \cite{BGT} and Section 5 of \cite{HSS}.

We fix throughout this section an $\mathbb{E}_\infty$-algebra $\mathcal{E}$ in $\mathsf{Cat}^{\on{perf}}$ in which every object is dualizable. And then we consider the $\mathcal{E}$-linear categories
\[
\begin{gathered}
\mathsf{Cat}^{\on{perf}}(\mathcal{E}):=\on{Mod}_\mathcal{E}(\mathsf{Cat}^{\on{perf}}),\\
\PrL(\mathcal{E}):=\on{Mod}_{\on{Ind}(\mathcal{E})}(\PrL_{st}).
\end{gathered}\]
\begin{ex}
The category of finite spectra $\on{Sp}^\omega$ is the unit of the symmetric monoidal structure on $\mathsf{Cat}^{\on{perf}}$. Hence the $\on{Sp}^\omega$-linear categories read:
\[
\begin{gathered}
\mathsf{Cat}^{\on{perf}}(\on{Sp}^\omega)=\mathsf{Cat}^{\on{perf}},\\
\PrL(\on{Sp}^\omega)=\PrL.
\end{gathered}\]
In particular, when this is the case, the category of $\mathcal{E}$-linear localizing motives reduces to the classical one in the sense of \cite{BGT}.
\end{ex}
\begin{rmk}
The first commutative diagram in Section \ref{compactibility} remains valid if we replace every category by its $\mathcal{E}$-linear version. However, this relies on our assumption that every object in $\mathcal{E}$ is dualizable. See \cite[Proposition 4.9, Proposition 4.10]{HSS} for the proof.
\end{rmk}
\begin{nota}
Note that $\mathsf{Cat}^{\on{perf}}(\mathcal{E})$ is a compactly generated category by \cite[Corollary 4.25]{BGT}. We denote by
\[
\psi: \mathsf{Cat}^{\on{perf}}(\mathcal{E}) \rightarrow \mathcal{P}sh(\mathsf{Cat}^{\on{perf}}(\mathcal{E})^\omega)
\]the Kan extension of the Yoneda embedding to the category of spectra-valued presheaves\[
\mathsf{Cat}^{\on{perf}}(\mathcal{E})^\omega \xrightarrow{j} Psh(\mathsf{Cat}^{\on{perf}}(\mathcal{E})^\omega) \xrightarrow{\Sigma_{+}^{\infty}} \mathcal{P}sh(\mathsf{Cat}^{\on{perf}}(\mathcal{E})^\omega).\]
\end{nota}
We say that a sequence $\mathcal{A}\xrightarrow{f}\mathcal{B}\xrightarrow{g}\mathcal{C}$ in $\mathsf{Cat}^{\on{perf}}$ is a Karoubi sequence if it is fiber and cofiber. Alternatively, that being said, $f$ is fully faithful and the induced map $\mathcal{B}/\mathcal{A}\rightarrow\mathcal{C}$ is an idempotent completion. And we say that a sequence $\mathcal{A}\xrightarrow{f}\mathcal{B}\xrightarrow{g}\mathcal{C}$ in $\mathsf{Cat}^{\on{perf}}(\mathcal{E})$ is a Karoubi (exact) sequence, if the forgetful functor $\on{U}:\mathsf{Cat}^{\on{perf}}(\mathcal{E})\rightarrow \mathsf{Cat}^{\on{perf}}$ sends the sequence to a Karoubi sequence in $\mathsf{Cat}^{\on{perf}}$.

Let $\mathcal{S}_{\on{loc}}$ be the collection of morphisms in $\mathcal{P}sh(\mathsf{Cat}^{\on{perf}}(\mathcal{E})^\omega)$ of the form
\[
\begin{gathered}
0 \rightarrow \Sigma^n \psi(0), \\
\Sigma^n(\psi(\mathcal{B}) / \psi(\mathcal{A})) \rightarrow \Sigma^n \psi(\mathcal{C}),
\end{gathered}
\]
where $\mathcal{A} \rightarrow \mathcal{B} \rightarrow \mathcal{C}$ is a Karoubi sequence in $\mathsf{Cat}^{\on{perf}}(\mathcal{E})$ and $n \leq 0$. The category of localizing motives $\on{Mot}_{\on{loc}}(\mathcal{E})$ is defined to be the full subcategory of $\mathcal{P}sh(\mathsf{Cat}^{\on{perf}}(\mathcal{E})^\omega)$ spanned by the $\mathcal{S}_{\on{loc}}$-local objects. $\on{Mot}_{\on{loc}}(\mathcal{E})$ is in fact a presentable category:
\begin{lm}(\cite[Proposition 5.15]{HSS})
There is a small set of morphisms $\mathcal{S}'_{\on{loc}}$ generates $\mathcal{S}_{\on{loc}}$ under filtered colimits. As a consequence, $\on{Mot}_{\on{loc}}(\mathcal{E})$ is an exact accessible localization of $\mathcal{P}sh(\mathsf{Cat}^{\on{perf}}(\mathcal{E})^\omega)$. In particular, $\on{Mot}_{\on{loc}}(\mathcal{E})$ is a stable presentable category. 
\end{lm}
We denote by $\mathcal{U}$ the composition\[
\mathcal{U}: \mathsf{Cat}^{\on{perf}}(\mathcal{E}) \xrightarrow{\psi} \mathcal{P}sh(\mathsf{Cat}^{\on{perf}}(\mathcal{E})^\omega)\xrightarrow {\on{L}} \on{Mot}_{\on{loc}}(\mathcal{E}).
\]
\begin{defn}(Localizing invariant)
Let $\mathcal{D}$ be a stable presentable category and let $F: \mathsf{Cat}^{\on{perf}}(\mathcal{E})\rightarrow \mathcal{D}$ be a functor. We say that $F$ is a localizing invariant if the following conditions are satisfied:
\begin{enumerate}
\item $F$ preserves filtered colimits.
\item $F$ preserves zero objects.
\item $F$ sends Karoubi sequences in $\mathsf{Cat}^{\on{perf}}(\mathcal{E})$ to cofiber sequences in $\mathcal{D}$.
\end{enumerate}
We denote by $\on{Fun}_{\on{loc}}(\mathsf{Cat}^{\on{perf}}(\mathcal{E}), \mathcal{D})$ the category of localizing invariants with values in $\mathcal{D}$.
\end{defn}
By construction, $\on{Mot}_{\on{loc}}(\mathcal{E})$ is equipped with a symmetric monoidal structure compatible with $\mathsf{Cat}^{\on{perf}}(\mathcal{E})$. Now we give a characterization of $\mathcal{U}$ as the universal localizing invariant and the universal symmetric monoidal localizing invariant which follows directly from the construction.
\begin{lm}(\cite[Theorem 5.17]{HSS})
The functor $\mathcal{U}:\mathsf{Cat}^{\on{perf}}(\mathcal{E})\rightarrow \on{Mot}_{\on{loc}}(\mathcal{E})$ is the universal localizing invariant in the sense that for every presentable stable category $\mathcal{D}$, the functor $\mathcal{U}$ induces an equivalence\[
\on{Fun}^{\on{L}}(\on{Mot}_{\on{loc}}(\mathcal{E}),\mathcal{D})\simeq \on{Fun}_{\on{loc}}(\mathsf{Cat}^{\on{perf}}(\mathcal{E}),\mathcal{D}).
\]
\end{lm}
\begin{lm}(\cite[Theorem 5.18]{HSS})\label{u.loc.inv.}
The symmetric monoidal functor $\mathcal{U}:\mathsf{Cat}^{\on{perf}}(\mathcal{E})\rightarrow \on{Mot}_{\on{loc}}(\mathcal{E})$ is the universal symmetric monoidal localizing invariant in the sense that for every presentably symmetric monoidal stable category $\mathcal{D}$, the functor $\mathcal{U}$ induces an equivalence\[
\on{Fun}^{\on{L},\otimes}(\on{Mot}_{\on{loc}}(\mathcal{E}),\mathcal{D})\simeq \on{Fun}^{\otimes}_{\on{loc}}(\mathsf{Cat}^{\on{perf}}(\mathcal{E}),\mathcal{D}).
\]
\end{lm}
We will use the following important theorem, known as the corepresentability of nonconnective K-theory. Before stating the theorem, we mention that this also relies heavily on the assumption that $\mathcal{E}$ consists of dualizable objects. For $\mathcal{E}=\on{Sp}^\omega$, the assumption is satisfied and this reduces to the classical case in \cite{BGT}.
\begin{defn}
Let $\mathcal{E}$ be an $\mathbb{E}_\infty$-algebra on $\mathsf{Cat}^{\on{perf}}$ consisting of dualizable objects. Let $\mathcal{A}$ be an $\mathcal{E}$-linear small, stable, idempotent complete category. We say $\mathcal{A}$ is smooth (resp. proper) if $\on{Ind}(\mathcal{A})$ is smooth (resp. proper) as a compactly generated stable category.
\end{defn}
\begin{prop}(\cite[Theorem 5.25]{HSS})
Let $\mathcal{A}$ and $\mathcal{B}$ be small, stable, idempotent complete categories and assume that $\mathcal{B}$ is smooth and proper. Then there is an equivalent of spectra:
\[\on{Mot}_{\on{loc}}(\mathcal{E})(\mathcal{U}(\mathcal{B}),\mathcal{U}(\mathcal{A}))\simeq \on{K}(\mathcal{B}^{op}\otimes_{\mathcal{E}} \mathcal{A}),
\]
where $\on{K}(-)$ is the nonconnective K-theory spectrum.
\end{prop}

\section{Construction of Comparison Functors of Motivic Homotopy Theories}\label{section:pres}
\subsection{Comparison between $\mathcal{SH}^{\vee}$ and $\on{Mot}_{\on{loc}}^{\mathbb{A}^1}$}
In this section, we fix a qcqs base scheme $S$.

Recall that for a $\mathbb{E}_\infty$-algebra $\mathcal{E}$ in which every object is dualizble, the $\mathcal{E}$-linear category of localizing motives $\on{Mot}_{\on{loc}}(\mathcal{E})$ is defined as the $\mathcal{S}_{\on{loc}}$-localization of $\mathcal{P}sh(\mathsf{Cat}^{\on{perf}}(\mathcal{E})^\omega)$. For our purpose of comparing localizing motives and $\mathcal{SH}$, it is reasonable to further localize $\on{Mot}_{\on{loc}}$ to get a category of $\mathbb{A}^1$-invariant localizing motives. And in the new category we have the corepresentability of $\mathbb{A}^1$-homotopy K-theory.

Note first that $\on{perf}_S$ consists of dualizable objects, since perfect complexes are precisely the dualizable objects in the category of quasi-coherent sheaves, as discussed in \cite{TT}, or Theorem 8.3 of \cite{Sch}. So we set $\mathcal{E}=\on{perf}_S$ as in Section \ref{E-linear.NC} and work $\on{perf}_S$-linearly. We begin by considering the natural functor\[
\mathcal{U}\circ \on{perf} : Sm^{\on{fp}}(S)^{op}\rightarrow \mathsf{Cat}^{\on{perf}}(\on{perf}_S)\rightarrow\on{Mot}_{\on{loc}}(\on{perf}_S).\]

For each $\mathcal{C}\in\mathsf{Cat}^{\on{perf}}(\on{perf}_S)^\omega$, we add to $\mathcal{S}_{\on{loc}}$ morphisms of the form\[
\psi(\mathcal{C})\rightarrow\psi(\on{perf}_{\mathbb{A}^1}\otimes\mathcal{C}),\]
and we denote the larger collection of morphisms by $\mathcal{S}_{\on{loc}}^{\mathbb{A}^1}$. Since only a small set of morphisms are added, the localization with respect to $\mathcal{S}_{\on{loc}}^{\mathbb{A}^1}$ exists and is still presentable. We denote simply by $\on{Mot}(S)_{\on{loc}}^{\mathbb{A}^1}$ the further localized $\mathcal{U}(\on{perf}_S)$-linear category $\on{Mot}_{\on{loc}}^{\mathbb{A}^1}(\on{perf}_S)$. It is immediate from Lemma \ref{u.loc.inv.} and the construction that $\on{Mot}(S)_{\on{loc}}^{\mathbb{A}^1}$ is the category of $\mathbb{A}^1$-invariant localizing motives:
\begin{lm}
The symmetric monoidal functor $\on{L}_{\mathbb{A}^1}\circ\mathcal{U}:\mathsf{Cat}^{\on{perf}}(\on{perf}_S)\rightarrow \on{Mot}(S)_{\on{loc}}^{\mathbb{A}^1}$ is the universal $\mathbb{A}^1$-invariant symmetric monoidal localizing invariant in the sense that for every presentably symmetric monoidal stable category $\mathcal{D}$, the functor $\on{L}_{\mathbb{A}^1}\circ\mathcal{U}$ induces an equivalence\[
\on{Fun}^{\on{L},\otimes}(\on{Mot}(S)_{\on{loc}}^{\mathbb{A}^1},\mathcal{D})\simeq \on{Fun}^\otimes_{\on{loc},\mathbb{A}^1}(\mathsf{Cat}^{\on{perf}}(\on{perf}_S),\mathcal{D}).
\]
In the following, we denote this universal $\mathbb{A}^1$-invariant localizing invariant simply by $\mathcal{U}$.
\end{lm}
We have the desired corepresentability of homotopy K-theory:
\begin{prop}\label{KH}
Let $\mathcal{A}$ and $\mathcal{B}$ be $\on{perf}_S$-linear small, stable, idempotent complete categories and assume that $\mathcal{B}$ is smooth and proper. Then there is an equivalence of spectra:
\[\on{Mot}(S)_{\on{loc}}^{\mathbb{A}^1}(\mathcal{U}(\mathcal{B}),\mathcal{U}(\mathcal{A}))\simeq \on{KH}(\mathcal{B}^{op}\otimes_{\on{perf}_S} \mathcal{A}),
\]
where $\on{KH}(-)$ is the $\mathbb{A}^1$-invariant homotopy K-theory spectrum.
\end{prop}
\begin{proof}
Denote by $\on{L}_{\mathbb{A}^1}$ the localization functor\[
\on{L}_{\mathbb{A}^1}: \on{Mot}(S)_{\on{loc}}\rightarrow\on{Mot}(S)_{\on{loc}}^{\mathbb{A}^1}.\]
Use the formula for the $\mathbb{A}^1$-localization of a finitary localization invariant in 1.14.2 of \cite{Sos}, we get\[
\begin{aligned}
\on{Mot}(S)_{\on{loc}}^{\mathbb{A}^1}(\on{L}_{\mathbb{A}^1}\circ\mathcal{U}(\mathcal{B}),\on{L}_{\mathbb{A}^1}\circ\mathcal{U}(\mathcal{A}))&\simeq\on{Mot}(S)_{\on{loc}}(\mathcal{U}(\mathcal{B}),\on{L}_{\mathbb{A}^1}\circ\mathcal{U}(\mathcal{A}))\\
&\simeq\on{Mot}(S)_{\on{loc}}(\mathcal{U}(\mathcal{B}),\underset{[n] \in \Delta^{\mathrm{op}}}{colim}\mathcal{U}(\mathcal{A} \otimes_{\on{perf}_S} \on{perf}_{\mathbb{A}^n}))\\
&\simeq\on{K}(\mathcal{B}^{op}\otimes_{\on{perf}_S}\underset{[n] \in \Delta^{\mathrm{op}}}{colim}\mathcal{A} \otimes_{\on{perf}_S} \on{perf}_{\mathbb{A}^n})\\
&\simeq\underset{[n] \in \Delta^{\mathrm{op}}}{colim}\on{K}(\mathcal{B}^{op}\otimes_{\on{perf}_S}\mathcal{A} \otimes_{\on{perf}_S} \on{perf}_{\mathbb{A}^n})\\
&\simeq\on{KH}(\mathcal{B}^{op}\otimes_{\on{perf}_S} \mathcal{A}).
\end{aligned}\]
\end{proof}
The goal is to extend the functor\[
\begin{aligned}
\phi: Sm^{\on{fp}}(S)^{op} & \rightarrow \on{Mot}(S)_{\on{loc}}^{\mathbb{A}^1}\\
X & \mapsto \mathcal{U}(\on{perf}_X)
\end{aligned}\]
to a left adjoint functor\[
\Phi: \mathcal{SH}(S)^{\vee}\rightarrow \on{Mot}(S)_{\on{loc}}^{\mathbb{A}^1}.\] 
To do this, we will realize the dual category $\mathcal{SH}(S)^{\vee}$ as a category of cosheaves, with $\mathbb{P}^1$ formally inverted. As we will see, such realization is naturally equipped with a universal property allowing us to extend certain functors, including $\phi$.

\subsubsection{Computation of $\mathcal{SH}(S)^{\vee}$}
We note first that in the construction of $\mathcal{SH}(S)$, instead of beginning from the anima-valued presheaf category $Psh(Sm^{\on{fp}}(S))$, we could also begin from the spectra-valued presheaf category $\mathcal{P}sh(Sm^{\on{fp}}(S))$. In fact, denote by $\Sigma^\infty_+$ the stabilization functor, then we have the following commutative diagram,
\[\begin{tikzcd}
	{Psh(Sm^{\on{fp}}(S))} && {\mathcal{P}sh(Sm^{\on{fp}}(S))} \\
	\\
	{Shv_{\on{Nis},\mathbb{A}^1}(Sm^{\on{fp}}(S))} && {\mathcal{S}hv_{\on{Nis},\mathbb{A}^1}(Sm^{\on{fp}}(S))} \\
	\\
	{\mathcal{SH}(S)} && {\mathcal{S}hv_{\on{Nis},\mathbb{A}^1}(Sm^{\on{fp}}(S))[(\mathbb{P}^1)^{-1}]}
	\arrow["{{\Sigma^\infty_+}}", from=1-1, to=1-3]
	\arrow["{{\on{L}_{\on{Nis},\mathbb{A}^1}}}", from=1-1, to=3-1]
	\arrow["{{\on{L}_{\on{Nis},\mathbb{A}^1}}}", from=1-3, to=3-3]
	\arrow["{{\Sigma^\infty_+}}", from=3-1, to=3-3]
	\arrow["{{\on{L}_{\mathbb{P}^1}\circ (-)_+}}", from=3-1, to=5-1]
	\arrow["{{\on{L}_{\mathbb{P}^1}}}", from=3-3, to=5-3]
	\arrow["{{\Sigma^\infty}}", from=5-1, to=5-3].
\end{tikzcd}\]
Here and in the following, the notation $Psh$ (resp. $Shv$) stands for anima-valued presheaf (resp. sheaf) category, while $\mathcal{P}sh$ (resp. $\mathcal{S}hv$) stands for spectra-valued presheaf (resp. sheaf) category. The commutativity of the upper square follows the fact that the Nisnevich excision and $\mathbb{A}^1$-invariance descent along the functor $\Sigma^\infty_+$. And the commutativity of the lower square follows from
the commutativity of two successively formal inversions, see \cite[Remark 4.11]{Rob}. On the other hand, since $\mathcal{SH}(S)$ is already stable, the bottom $\Sigma^\infty$ is an equivalence.

By the discussion in Section \ref{compactibility} and Section \ref{SH}, every category appeared in the diagram is presentable, compactly generated, and every vertical functor is strongly left adjoint. We will compute the image of the sequence\[
\mathcal{P}sh(Sm^{\on{fp}}(S))\xrightarrow{\on{L}_{\on{Nis},\mathbb{A}^1}}\mathcal{S}hv_{\on{Nis},\mathbb{A}^1}(Sm^{\on{fp}}(S))\xrightarrow{\on{L}_{\mathbb{P}^1}}\mathcal{S}hv_{\on{Nis},\mathbb{A}^1}(Sm^{\on{fp}}(S))[(\mathbb{P}^1)^{-1}]
\]
in $\PrL_{cg}$ under the functor $(-)^{\vee}:\PrL_{cg}\rightarrow\PrL_{cg}$.
\begin{lm}\label{id1}
The dual of $\mathcal{P}sh(Sm^{\on{fp}}(S))$ is $\mathcal{P}sh(Sm^{\on{fp}}(S)^{op})$.
\end{lm}
\begin{proof}
This follows a direct computation 
\[\mathcal{P}sh(Sm^{\on{fp}}(S))^{\vee}=\on{Fun}^L(\mathcal{P}sh(Sm^{\on{fp}}(S)),\on{Sp})=\on{Fun}(Sm^{\on{fp}}(S),\on{Sp})=\mathcal{P}sh(Sm^{\on{fp}}(S)^{op}).\]
\end{proof}
\begin{lm}\label{id2}
The composition $\on{L}_{\on{Nis},\mathbb{A}^1}\circ y: Sm^{\on{fp}}(S)\rightarrow\mathcal{S}hv_{\on{Nis},\mathbb{A}^1}(Sm^{\on{fp}}(S))$ induces a fully faithful map:\[
\on{Fun}^L(\mathcal{S}hv_{\on{Nis},\mathbb{A}^1}(Sm^{\on{fp}}(S)),\on{Sp})\rightarrow\on{Fun}(Sm^{\on{fp}}(S),\on{Sp}).\]
The essential image is precisely the subcategory spanned by functors $F: Sm^{\on{fp}}(S)\rightarrow\on{Sp}$ satisfying $\mathbb{A}^1$-invariance, sending empty scheme $\emptyset$ to the zero object, and sending Nisnevich squares to pushouts. We denote the essential image by $\on{Fun}_{\on{Nis},\mathbb{A}^1}(Sm^{\on{fp}}(S),\on{Sp})$.
\end{lm}
\begin{proof}
Recall that whether a presheaf lying in $\mathcal{S}hv_{\on{Nis},\mathbb{A}^1}(Sm^{\on{fp}}(S))$ is detected fully by Nisnevich excision and $\mathbb{A}^1$-invariance. Alternatively, $\mathcal{S}hv_{\on{Nis},\mathbb{A}^1}(Sm^{\on{fp}}(S))$ is spanned by the $\mathcal{S}$-local objects of $\mathcal{P}sh(Sm^{\on{fp}}(S))$, where $\mathcal{S}$ is the set of morphisms
\begin{itemize}
\item $0\rightarrow y(\emptyset)$, where $0$ is the zero object;
\item for each Nisnevich square\[\begin{tikzcd}
	W && U \\
	\\
	V && X
	\arrow[from=1-1, to=1-3]
	\arrow[from=1-1, to=3-1]
	\arrow[from=1-3, to=3-3]
	\arrow[from=3-1, to=3-3],
\end{tikzcd}\]
$y(U)\amalg_{y(W)}y(V)\rightarrow y(X)$;
\item for each $X\in Sm^{\on{fp}}(S)$, $y(X\times\mathbb{A}^1)\rightarrow y(X)$.
\end{itemize}
Then we consider the equivalence\[
\on{res}: \on{Fun}^L(\mathcal{P}sh(Sm^{\on{fp}}(S)),\on{Sp})\leftrightarrows\on{Fun}(Sm^{\on{fp}}(S),\on{Sp}): \on{Lan}.\]
We claim that the subcategory of $\on{Fun}^L(\mathcal{P}sh(Sm^{\on{fp}}(S)),\on{Sp})$ spanned by functors that factor through $\mathcal{S}hv_{\on{Nis},\mathbb{A}^1}(Sm^{\on{fp}}(S))$ corresponds to the essential image described as in the lemma. By \cite[Proposition 5.5.4.20]{HTT}, this subcategory on the left hand side is spanned precisely by the colimits-preserving functors that invert all morphisms in $\mathcal{S}$. But it is easy to check that a functor $F:Sm^{\on{fp}}(S)\rightarrow\on{Sp}$ is $\mathbb{A}^1$-invariant, preserves initial object, and sends Nisnevich squares to pushouts if and only if its left Kan extension $\on{Lan}F$ sends all $s\in\mathcal{S}$ to equivalences.
\end{proof}
We also call the dual of category of sheaves category of cosheaves. In this case we use the notation\[
co\mathcal{S}hv_{\on{coNis},co{\text -}\mathbb{A}^1}(Sm^{\on{fp}}(S)):=\mathcal{S}hv_{\on{Nis},\mathbb{A}^1}(Sm^{\on{fp}}(S))^\vee.\]
\begin{rmk}\label{coshv.B.loc}
Classically, for a presentable category $\mathcal{E}$ and a site $(\mathcal{C},\tau)$, we define the category of $\mathcal{E}$-valued cosheaves as the opposite category of the category of $\mathcal{E}^{op}$-valued sheaves, i.e.,\[
coShv(\mathcal{C};\mathcal{E}):=Shv(\mathcal{C};\mathcal{E}^{op})^{op}.\]
If the $\mathcal{E}$-valued sheaf category is dualizable, this classical notion of category of cosheaves is the same as ours, since as subcategories of presheaf category, they are checked by exactly the same codescent conditions.
\end{rmk}
\begin{lm}\label{desc}
Via the equivalences in Lemma \ref{id1} and Lemma \ref{id2}, the functor\[
\on{Fun}(Sm^{\on{fp}}(S),\on{Sp})\rightarrow\on{Fun}_{\on{Nis},\mathbb{A}^1}(Sm^{\on{fp}}(S),\on{Sp})\]
induced by\[
\on{L}_{\on{Nis},\mathbb{A}^1}^\vee:\on{Fun}^L(\mathcal{P}sh(Sm^{\on{fp}}(S)),\on{Sp})\rightarrow\on{Fun}^L(\mathcal{S}hv_{\on{Nis},\mathbb{A}^1}(Sm^{\on{fp}}(S)),\on{Sp})\]
is a Bousfield localization with respect to $\mathcal{S}$ consisting of maps
\begin{itemize}
\item $Sm^{\on{fp}}(S)(\emptyset,-)\rightarrow0$;
\item for each Nisnevich square\[\begin{tikzcd}
	W && U \\
	\\
	V && X
	\arrow[from=1-1, to=1-3]
	\arrow[from=1-1, to=3-1]
	\arrow[from=1-3, to=3-3]
	\arrow[from=3-1, to=3-3],
\end{tikzcd}\]
$Sm^{\on{fp}}(S)(U,-)\amalg_{Sm^{\on{fp}}(S)(X,-)}Sm^{\on{fp}}(S)(V,-)\rightarrow Sm^{\on{fp}}(S)(W,-)$;
\item for each $X\in Sm^{\on{fp}}(S)$, $Sm^{\on{fp}}(S)(X,-)\rightarrow Sm^{\on{fp}}(S)(X\times\mathbb{A}^1,-)$.
\end{itemize}
\end{lm}
\begin{proof}
By a similar argument, we see that $\on{Fun}_{\on{Nis},\mathbb{A}^1}(Sm^{\on{fp}}(S),\on{Sp})$ is indeed the subcategory of $\mathcal{S}$-local objects. On the other hand, the right adjoint\[
\on{Fun}_{\on{Nis},\mathbb{A}^1}(Sm^{\on{fp}}(S),\on{Sp})\rightarrow\on{Fun}(Sm^{\on{fp}}(S),\on{Sp})\]
is induced by the precomposition functor\[
\on{L}_{\on{Nis},\mathbb{A}^1}^\ast:\on{Fun}^L(\mathcal{S}hv_{\on{Nis},\mathbb{A}^1}(Sm^{\on{fp}}(S)),\on{Sp})\rightarrow\on{Fun}^L(\mathcal{P}sh(Sm^{\on{fp}}(S)),\on{Sp}).\]
Applying identifications Lemma \ref{id1} and Lemma \ref{id2} again, we see that it is fully faithful.
\end{proof}
\begin{rmk}
In the proof of the lemma above, it is important that we work with (co)sheaves valued in a stable category, hence a square is a pullback if and only if it is a pushout. The techniques involved in this and next sections work equally for stable-category-valued sheaves on a site generated by a cd-structure. A rough idea is that, while we want a universal characterization of the form \ref{cons}, the category of cosheaves most naturally admits a dual form of universal property. The stability and cd-structure-generation assumptions blur this difference.
\end{rmk}
Finally, since $\mathbb{P}^1$ is compact in $\mathcal{S}hv_{\on{Nis},\mathbb{A}^1}(Sm^{\on{fp}}(S))$, by the compatibility of taking dual and formal inversion in Section \ref{compactibility}, we have\[
\mathcal{SH}(S)^{\vee}\simeq\mathcal{S}hv_{\on{Nis},\mathbb{A}^1}(Sm^{\on{fp}}(S))[(\mathbb{P}^1)^{-1}]^{\vee}\simeq co\mathcal{S}hv_{\on{coNis},co{\text -}\mathbb{A}^1}(Sm^{\on{fp}}(S))[((\mathbb{P}^1)^{dual})^{-1}].\]

\subsubsection{Extending the functor $\phi$}
In the above section, we compute the dual of $\mathcal{SH}(S)$: It is the formal inversion with respect to $(\mathbb{P}^1)^{dual}$ of the full subcategory of $\mathcal{P}sh((Sm^{\on{fp}}(S))^{op})$ spanned by functors satisfying Nisnevich coexcision and being co-$\mathbb{A}^1$-invariant. In this section, we use the universal properties arising naturally from the computations in the last section to extend the obvious functor $\phi: Sm^{\on{fp}}(S)^{op}\rightarrow \on{Mot}(S)_{\on{loc}}^{\mathbb{A}^1}$ to a functor $\Phi: \mathcal{SH}(S)^{\vee}\rightarrow \on{Mot}(S)_{\on{loc}}^{\mathbb{A}^1}$. Step by step, we follow the diagram:
\[\begin{tikzcd}
	{Sm^{\on{fp}}(S)^{op}} &&& {\on{Mot}(S)_{\on{loc}}^{\mathbb{A}^1}} \\
	{\mathcal{P}sh(Sm^{\on{fp}}(S)^{op})} \\
	{co\mathcal{S}hv_{\on{coNis},co{\text -}\mathbb{A}^1}(Sm^{\on{fp}}(S))} \\
	{co\mathcal{S}hv_{\on{coNis},co{\text -}\mathbb{A}^1}(Sm^{\on{fp}}(S))[((\mathbb{P}^1)^{dual})^{-1}]}
	\arrow["\phi", from=1-1, to=1-4]
	\arrow["j"', from=1-1, to=2-1]
	\arrow[from=2-1, to=1-4]
	\arrow["{\on{L}_{\on{coNis},\on{co}{\text -}\mathbb{A}^1}}"', from=2-1, to=3-1]
	\arrow[from=3-1, to=1-4]
	\arrow["{\on{L}_{(\mathbb{P}^1)^{dual}}}"', from=3-1, to=4-1]
	\arrow["\Phi"', from=4-1, to=1-4]
\end{tikzcd}\]

Note first that $\on{Mot}(S)_{\on{loc}}^{\mathbb{A}^1}$ is a presentable stable category. The spectral Yoneda embedding induces an equivalence \[\on{Fun}^{L}(\mathcal{P}sh(Sm^{\on{fp}}(S)^{op}), \on{Mot}(S)_{\on{loc}}^{\mathbb{A}^1})\simeq\on{Fun}(Sm^{\on{fp}}(S)^{op},\on{Mot}(S)_{\on{loc}}^{\mathbb{A}^1}).\]

As discussed in Lemma \ref{desc}, $co\mathcal{S}hv_{\on{Nis},\mathbb{A}^1}(Sm^{\on{fp}}(S))$ is the subcategory of $\mathcal{P}sh(Sm^{\on{fp}}(S)^{op})$ spanned by the $\mathcal{S}$-local objects. Applying a similar argument as in the last section again, the universal property of the category of cosheaves reads:
\begin{lm}\label{coShv.universal.property}
The composition of Yoneda embedding and localization induces an equivalence of functor categories\[
\on{Fun}^{L}(co\mathcal{S}hv_{\on{Nis},\mathbb{A}^1}(Sm^{\on{fp}}(S)), \on{Mot}(S)_{\on{loc}}^{\mathbb{A}^1})\simeq\on{Fun}_{\on{Nis},\mathbb{A}^1}(Sm^{\on{fp}}(S)^{op},\on{Mot}(S)_{\on{loc}}^{\mathbb{A}^1}),\]
where the subcategory on the right hand side is spanned by all the $\mathbb{A}^1$-invariant functors preserving terminal object (or, sending the empty scheme $\emptyset$ to the zero object) and sending Nisnevich squares to pushouts.
\end{lm}
\begin{rmk}
Again, the above computations rely essentially on that we're working with stable-category-valued sheaves.
\end{rmk}
To extend the universal property described in Lemma \ref{coShv.universal.property} to the formal inversion 
\[co\mathcal{S}hv_{\on{coNis},co{\text -}\mathbb{A}^1}(Sm^{\on{fp}}(S))[((\mathbb{P}^1)^{dual})^{-1}],\] we notice first that $(\mathbb{P}^1)^{dual}$ admits the following internal characterization:
\begin{lm}\label{description}
$(\mathbb{P}^1)^{dual}\in co\mathcal{S}hv_{\on{coNis},co{\text -}\mathbb{A}^1}(Sm^{\on{fp}}(S))$ identifies with $fib(\on{L}\circ j(\mathbb{P}^1)\xrightarrow{\on{L}\circ j(\infty)}\on{L}\circ j(S))$, where the fiber is induced as in the following diagram:
\[\begin{tikzcd}
	{Sm^{\operatorname{fp}}(S)^{op}} && {S\xrightarrow{\infty}\mathbb{P}^1} \\
	{\mathcal{P}sh(Sm^{\operatorname{fp}}(S)^{op})} && {j(\mathbb{P}^1)\xrightarrow{j(\infty)}j(S)} \\
	{co\mathcal{S}hv_{\operatorname{coNis},co{\text -}\mathbb{A}^1}(Sm^{\operatorname{fp}}(S))} && {\operatorname{L}\circ j(\mathbb{P}^1)\xrightarrow{\operatorname{L}\circ j(\infty)}\operatorname{L}\circ j(S)}
	\arrow["j", from=1-1, to=2-1]
	\arrow[maps to, from=1-3, to=2-3]
	\arrow["{\operatorname{L}_{\operatorname{coNis},\operatorname{co}{\text -}\mathbb{A}^1}}", from=2-1, to=3-1]
	\arrow[maps to, from=2-3, to=3-3].
\end{tikzcd}\]
\end{lm}
\begin{proof}
By definition, the object $\mathbb{P}^1\in\mathcal{S}hv_{\on{Nis},\mathbb{A}^1}(Sm^{\on{fp}}(S))$ corresponds to a left adjoint functor\[
\on{L}_{\on{Nis},\mathbb{A}^1}\circ (\mathbb{P}^1,\infty):\on{Sp}\rightarrow\mathcal{P}sh(Sm^{\on{fp}}(S))\rightarrow\mathcal{S}hv_{\on{Nis},\mathbb{A}^1}(Sm^{\on{fp}}(S)).\]
On the other hand, $(\mathbb{P}^1,\infty)\in\mathcal{P}sh(Sm^{\on{fp}}(S))$ is equivalent to the cofiber of\[
j(S)\xrightarrow{j(\infty)}j(\mathbb{P}^1).\]
Combining them together, we have\[
(\mathbb{P}^1)^{dual}\simeq\on{L}_{\on{coNis},\on{co}{\text -}\mathbb{A}^1}(fib(j(\mathbb{P}^1)^{dual}\rightarrow j(S)^{dual})).\]
Now the claim follows that $\on{L}_{\on{coNis},\on{co}{\text -}\mathbb{A}^1}$ commutes with finite limits and the next lemma.
\end{proof}
\begin{lm}
Let $\mathcal{C}$ be a small, stable category and let $X$ be an object in $\mathcal{C}$. We denote by $j(X)$ the Yoneda embedding of $X$ in $\mathcal{P}sh(\mathcal{C})$. Since $j(X)$ is compact, it corresponds to an object $j(X)^{dual}$ in the dual category. Regarded as a functor, \[j(X)^{dual}:\mathcal{C}\rightarrow\on{Sp}\]
is simply the covariant representable functor\[
\on{Map}_{\mathcal{C}}(X,-).\]
\end{lm}
\begin{proof}
By definition, $j(X)^{dual}\in\on{Fun}^L(\mathcal{P}sh(\mathcal{C}),\on{Sp})$ is the right adjoint of the functor\[
X:\on{Sp}\rightarrow\mathcal{P}sh(\mathcal{C}),\]
which is the left adjoint of the left Kan extension of\[
\on{Map}_{\mathcal{C}}(X,-):\mathcal{C}\rightarrow\on{Sp}.\]
\end{proof}
Combining Lemma \ref{coShv.universal.property} and the Lemma \ref{description}, we get the desired universal property of $\mathcal{SH}(S)^\vee$.
\begin{prop}\label{cons}
The composition $\on{L}_{(\mathbb{P}^1)^{dual}}\circ\on{L}_{\on{coNis},\on{co}{\text -}\mathbb{A}^1}\circ j$ induces, for each presentable, stable category $\mathcal{D}$, a fully faithful functor\[
\on{Fun}^L(\mathcal{SH}(S)^\vee,\mathcal{D})\rightarrow\on{Fun}(Sm^{\on{fp}}(S)^{op},\mathcal{D})
\]
with the essential image spanned by $\mathbb{A}^1$-invariant functors $F$ that preserve terminal object, send Nisnevich squares to pushouts, and send the fiber\[fib(F(\mathbb{P}^1)\xrightarrow{F(\infty)}F(S))\]
to an invertible object.

Moreover, since the symmetric monoidal structures are compatible, the claim remains valid for symmetric monoidal category $\mathcal{D}$ and symmetric monoidal functors.

In particular, we have\[\on{Fun}^{L,\otimes}(\mathcal{SH}(S)^\vee,\on{Mot}(S)_{\on{loc}}^{\mathbb{A}^1})\simeq\on{Fun}^{\otimes}_{\on{Nis},\mathbb{A}^1,\mathbb{P}^1}(Sm^{\on{fp}}(S)^{op},\on{Mot}(S)_{\on{loc}}^{\mathbb{A}^1})
\]
\end{prop}
\begin{prop}\label{ver.}
The functor\[
\begin{aligned}
\phi: Sm^{\on{fp}}(S)^{op} &\rightarrow \on{Mot}(S)_{\on{loc}}^{\mathbb{A}^1}\\
 X & \mapsto \mathcal{U}(\on{perf}_X)
 \end{aligned}\]
satisfies all the conditions in the Proposition \ref{cons}. In particular, it extends to a symmetric monoidal left adjoint functor \[ \Phi:\mathcal{SH}(S)^\vee \rightarrow \on{Mot}(S)_{\on{loc}}^{\mathbb{A}^1}.
\]
\end{prop}
\begin{proof}
\begin{itemize}
\item Compatibility of symmetric monoidal structures.\\
Both of the $S$-linear functors\[
\begin{gathered}
\on{perf}:Sm(S)^{op}\rightarrow \mathsf{Cat}^{\on{perf}}(\on{perf}_S)\\
\mathcal{U}:\mathsf{Cat}^{\on{perf}}(\on{perf}_S)\rightarrow\on{Mot}(S)_{\on{loc}}^{\mathbb{A}^1}
\end{gathered}
\]
are symmetric monoidal. So is their composition.
\item Nisnevich excision.\\
It is obvious that the empty scheme is sent to $\mathcal{U}(\on{perf}_{\emptyset})=0$. On the other hand, recall that the functor\[
\on{perf}:Sm(S)^{op}\rightarrow \mathsf{Cat}^{\on{perf}}(\on{perf}_S)\]
satisfies fppf descent. In particular, for any Nisvevich square
\[\begin{tikzcd}
	W && V \\
	\\
	U && X
	\arrow[hook, from=1-1, to=1-3]
	\arrow[from=1-1, to=3-1]
	\arrow[from=1-3, to=3-3]
	\arrow[hook, from=3-1, to=3-3],
\end{tikzcd}\]
we have a pullback diagram
\[\begin{tikzcd}
	{\on{perf}_X} && {\on{perf}_U} \\
	\\
	{\on{perf}_V} && {\on{perf}_W}
	\arrow[from=1-1, to=1-3]
	\arrow[from=1-1, to=3-1]
	\arrow[from=1-3, to=3-3]
	\arrow[from=3-1, to=3-3].
\end{tikzcd}\]
For an open immersion $U\hookrightarrow X$, if we denote by $\on{perf}_{X,X-U}$ the fiber\[
\on{perf}_{X,X-U}\rightarrow\on{perf}_X\rightarrow\on{perf}_U,\]
then we have a commutative diagram
\[\begin{tikzcd}
	{\on{perf}_{X,X-U}} && {\on{perf}_X} && {\on{perf}_U} \\
	\\
	{\on{perf}_{V,V-W}} && {\on{perf}_V} && {\on{perf}_W}
	\arrow[from=1-1, to=1-3]
	\arrow["\simeq"', from=1-1, to=3-1]
	\arrow[from=1-3, to=1-5]
	\arrow[from=1-3, to=3-3]
	\arrow[from=1-5, to=3-5]
	\arrow[from=3-1, to=3-3]
	\arrow[from=3-3, to=3-5].
\end{tikzcd}\]
Since $\mathcal{U}$ is the universal localizing invariant, applying $\mathcal{U}$ gives a commutative diagram in which, again, the horizontal sequences are fibers and the left most arrow is an equivalence. As a consequence,
\[\begin{tikzcd}
	{\mathcal{U}(\on{perf}_X)} && {\mathcal{U}(\on{perf}_U)} \\
	\\
	{\mathcal{U}(\on{perf}_V)} && {\mathcal{U}(\on{perf}_W)}
	\arrow[from=1-1, to=1-3]
	\arrow[from=1-1, to=3-1]
	\arrow[from=1-3, to=3-3]
	\arrow[from=3-1, to=3-3]
\end{tikzcd}\]
is a pullback as desired.
\item $\mathbb{A}^1$-invariance.\\
This is true by impose.
\item $\mathbb{P}^1$-invertibility.\\
We have a cofiber sequence:
\[\begin{tikzcd}
	{\on{perf}_S} && {\on{perf}_{\mathbb{P}^1}} && {\on{perf}_S}
	\arrow[from=1-1, to=1-3]
	\arrow["\on{perf}_\infty", from=1-3, to=1-5].
\end{tikzcd}\]
And $\mathcal{U}$ sends it to a both fiber and cofiber sequence. In particular, the fiber of \[\phi(\mathbb{P}^1\xrightarrow{\infty} S)\] is the unit, thus is invertible.
\end{itemize}
\end{proof}

\subsection{Comparison between $\mathcal{MS}^\vee$ and $\mathsf{Mot}_\mathsf{loc}$}
Throughout this section, we fix $S$ to be a qcqs scheme. We will give a non-$\mathbb{A}^1$-invariant analogous story in this section. That is, we want to compute $\mathcal{MS}^\vee$ and then extend the natural functor 
\[\begin{aligned}
\psi: Sm^{\on{fp}}(S)^{op} &\rightarrow \on{Mot}(S)_{\on{loc}}\\
X&\mapsto \mathcal{U}(\on{perf}_X)
\end{aligned}\]
to a functor\[\Psi:\mathcal{MS}(S)^{\vee}\rightarrow\on{Mot}(S)_{\on{loc}}.\]

Since all the arguments in this section work mutatis mutandis as in the last section, we will mainly present only the claims in this section.
\subsubsection{Computation of $\mathcal{MS}^\vee$}
Recall that $\mathcal{MS}(S)$ is defined from the spectra-valued presheaf category $\mathcal{P}sh(Sm^{\on{fp}}(S))$ by first localizing with respect to Nisnevich excision and elementary blowup excision, and then inverting $\mathbb{P}^1$. And each step gives a compactly generated, presentable, stable category. We want to compute the image of the sequence\[
\mathcal{P}sh(Sm^{\on{fp}}(S))\xrightarrow{\on{L}_{\on{Nis},\on{ebu}}}\mathcal{P}sh_{\on{Nis},\on{ebu}}(Sm^{\on{fp}}(S))\xrightarrow{\on{L}_{\mathbb{P}^1}}\mathcal{P}sh_{\on{Nis},\on{ebu}}(Sm^{\on{fp}}(S))[(\mathbb{P}^1)^{-1}]
\]
in $\PrL_{cg}$ under the functor $(-)^{\vee}:\PrL_{cg}\rightarrow\PrL_{cg}$.
\begin{lm}\label{id1`}
The dual of $\mathcal{P}sh(Sm^{\on{fp}}(S))$ is $\mathcal{P}sh(Sm^{\on{fp}}(S)^{op})$.
\end{lm}
\begin{lm}\label{id2`}
The canonical map $\on{L}_{\on{Nis},\on{ebu}}\circ y: Sm^{\on{fp}}(S)\rightarrow\mathcal{P}sh_{\on{Nis},\on{ebu}}(Sm^{\on{fp}}(S))$ induces a fully faithful map:\[
\on{Fun}^L(\mathcal{P}sh_{\on{Nis},\on{ebu}}(Sm^{\on{fp}}(S)),\on{Sp})\rightarrow\on{Fun}(Sm^{\on{fp}}(S),\on{Sp}).\]
The essential image is precisely the subcategory spanned by functors $F: Sm^{\on{fp}}(S)\rightarrow\on{Sp}$ sending empty scheme $\emptyset$ to the zero object, and sending Nisnevich squares and elementary blowup squares to pushouts. We denote the essential image by $\on{Fun}_{\on{Nis},\on{ebu}}(Sm^{\on{fp}}(S),\on{Sp})$.
\end{lm}
\begin{nota}
We also call the dual of the category of presheaves satisfying certain excision conditions the category of presheaves satisfying corresponding coexcision conditions. In this case we use the notation\[
co\mathcal{P}sh_{\on{coNis},\on{coebu}}(Sm^{\on{fp}}(S)):=\mathcal{P}sh_{\on{Nis},\on{ebu}}(Sm^{\on{fp}}(S))^\vee.\]
\end{nota}
\begin{lm}\label{desc`}
Via the equivalences in Lemma \ref{id1`} and Lemma \ref{id2`}, the functor\[
\on{Fun}(Sm^{\on{fp}}(S),\on{Sp})\rightarrow\on{Fun}_{\on{Nis},\on{ebu}}(Sm^{\on{fp}}(S),\on{Sp})\]
induced by\[
\on{L}_{\on{Nis},\on{ebu}}^\vee:\on{Fun}^L(\mathcal{P}sh(Sm^{\on{fp}}(S)),\on{Sp})\rightarrow\on{Fun}^L(\mathcal{P}sh_{\on{Nis},\on{ebu}}(Sm^{\on{fp}}(S)),\on{Sp})\]
is a Bousfield localization with respect to a small set of maps $\mathcal{S}$.
\end{lm}
\begin{lm}
Since $\mathbb{P}^1$ is compact in $\mathcal{P}sh_{\on{Nis},\on{ebu}}(Sm^{\on{fp}}(S))$, we have\[
\mathcal{MS}(S)^{\vee}\simeq\mathcal{P}sh_{\on{Nis},\on{ebu}}(Sm^{\on{fp}}(S))[(\mathbb{P}^1)^{-1}]^{\vee}\simeq co\mathcal{P}sh_{\on{coNis},\on{coebu}}(Sm^{\on{fp}}(S))[((\mathbb{P}^1)^{dual})^{-1}].\]
\end{lm}

\subsubsection{Extending the functor $\psi$}
Consider the diagram:
\[\begin{tikzcd}
	{Sm^{\on{fp}}(S)^{op}} &&& {\operatorname{Mot}(S)_{\operatorname{loc}}} \\
	{\mathcal{P}sh(Sm^{\on{fp}}(S)^{op})} \\
	{co\mathcal{P}sh_{\operatorname{coNis},\operatorname{coebu}}(Sm^{\on{fp}}(S))} \\
	{co\mathcal{P}sh_{\operatorname{coNis},\operatorname{coebu}}(Sm^{\on{fp}}(S))[((\mathbb{P}^1)^{dual})^{-1}]}
	\arrow["\psi", from=1-1, to=1-4]
	\arrow["j"', from=1-1, to=2-1]
	\arrow[from=2-1, to=1-4]
	\arrow["{\on{L}_{\on{coNis},\on{coebu}}}"', from=2-1, to=3-1]
	\arrow[from=3-1, to=1-4]
	\arrow["{\operatorname{L}_{(\mathbb{P}^1)^{dual}}}"', from=3-1, to=4-1]
	\arrow["\Psi"', from=4-1, to=1-4]
\end{tikzcd}\]
We have the following non-$\mathbb{A}^1$-invariant analogue of Proposition \ref{cons} and Proposition \ref{ver.}:
\begin{prop}\label{cons`}
The composition $\on{L}_{(\mathbb{P}^1)^{dual}}\circ\on{L}_{\on{coNis},\on{coebu}}\circ j$ induces, for each presentable, stable category $\mathcal{D}$, a fully faithful functor\[
\on{Fun}^L(\mathcal{MS}(S)^\vee,\mathcal{D})\rightarrow\on{Fun}(Sm^{\on{fp}}(S)^{op},\mathcal{D})
\]
with the essential image spanned by functors $F$ that preserve terminal object, send Nisnevich squares and elementary blowup squares to pushouts, and send the fiber\[fib(F(\mathbb{P}^1)\xrightarrow{F(\infty)}F(S))\]
to an invertible object.

Moreover, since the symmetric monoidal structures are compatible, the claim remains valid for symmetric monoidal category $\mathcal{D}$ and monoidal functors.

In particular, we have\[\on{Fun}^{L,\otimes}(\mathcal{MS}(S)^\vee,\on{Mot}(S)_{\on{loc}})\simeq\on{Fun}^{\otimes}_{\on{Nis},\on{ebu},\mathbb{P}^1}(Sm^{\on{fp}}(S)^{op},\on{Mot}(S)_{\on{loc}})
\]
\end{prop}
\begin{prop}\label{ver.ms}
The functor\[
\begin{aligned}
\psi: Sm^{\on{fp}}(S)^{op} &\rightarrow \on{Mot}(S)_{\on{loc}}\\
 X & \mapsto \mathcal{U}(\on{perf}_X)
 \end{aligned}\]
satisfies all the conditions in the Proposition \ref{cons`}. In particular, it extends to a symmetric monoidal left adjoint functor \[ \Psi:\mathcal{MS}(S)^\vee \rightarrow \on{Mot}(S)_{\on{loc}}.
\]
\end{prop}
\begin{proof}
The functor $\psi$ satisfies elementary blowup excision by \cite[Theorem A]{Kha}. The other conditions are checked as in Proposition \ref{ver.}.
\end{proof}

\section{Fully-Faithfulness of the Factored Functors}
In this section, we apply the general strategy introduced in Section \ref{fac} to the symmetric monoidal left adjoint functor $\Phi$ constructed in Proposition \ref{ver.} and $\Psi$ constructed in Proposition \ref{ver.ms}. As discussed in Section \ref{fac}, there are several choices of categories of modules through which our comparison functors factor automatically. Among these choices, we will focus on the one over the commutative algebra $G(1)$, where $G$ is the right adjoint functor of $\Phi$ or $\Psi$. As we will see, in both cases, the factored functor is fully faithful when restricted on objects with good dualizability. This suggests us considering the rigid generation if we want a global fully-faithfulness result. Here one difference between $\mathcal{SH}$ and $\mathcal{MS}$ appears.
\subsection{Fully-faithfulness of the factored functor, $\mathbb{A}^1$-invariant case}
We begin by recalling definitions and results on the rigid generation of the category $\mathcal{SH}(S)$. We denote by
\begin{itemize}
\item $\mathcal{SH}^\omega(S)$ the full subcategory of $\mathcal{SH}(S)$ spanned by the compact objects, and
\item $\mathcal{SH}^{dual}(S)$ the full subcategory of $\mathcal{SH}(S)$ spanned by the dualizable objects.
\end{itemize}
Then, as proved in \cite{Rio}, there is an inclusion
\[
\mathcal{SH}^{dual}(S)\subset\mathcal{SH}^\omega(S).\]

In the case that $S=\on{Spec} k$ for some field $k$, we denote the exponential characteristic of $k$ by $e$. Then \cite[Corollary B.2]{LYZ+} tells us that for each prime $l$ coprime with $e$, localizing at $l$ makes the inclusion above an equivalence. In particular, if we denote by $\mathcal{SH}(k)[\frac{1}{e}]$ the full subcategory of $\mathcal{SH}(k)$ of $\mathcal{S}$-local objects, where $\mathcal{S}$ is the collection of morphisms $E\xrightarrow{\times e}E$ for all $E\in\mathcal{SH}(k)$, we then have\[
\mathcal{SH}^{dual}(k)[\frac{1}{e}]\simeq\mathcal{SH}^\omega(k)[\frac{1}{e}].\]

We observe that for any stable, compactly generated, presentably symmetric monoidal category $\mathcal{E}$, the compact objects and dualizable objects in $\mathcal{E}^{dual}$ correspond to those in $\mathcal{E}$ respectively. In particular, $\mathcal{SH}(k)[\frac{1}{e}]^\vee$ satisfies
\[(\mathcal{SH}(k)[\frac{1}{e}]^{\vee})^{dual}\simeq (\mathcal{SH}(k)[\frac{1}{e}]^{\vee})^\omega.\]

On the other hand, the functor \[\phi: \mathcal{SH}(k)^\vee\rightarrow\on{Mot}(k)_{\on{loc}}^{\mathbb{A}^1}\] induces naturally an $e$-invertible functor \[\phi: \mathcal{SH}(k)^{\vee}[\frac{1}{e}]\rightarrow\on{Mot}(k)_{\on{loc}}^{\mathbb{A}^1}[\frac{1}{e}].\]
The inverting $e$ operator and the taking dual operator commute with each other\[
\begin{aligned}
\mathcal{SH}(k)[\frac{1}{e}]^{\vee}&\simeq\on{Fun}^L(\mathcal{SH}(k)[\frac{1}{e}],\on{Sp})\\
&\simeq\on{Fun}^L(\mathcal{SH}(k),\on{Sp}[\frac{1}{e}])\\
&\simeq\on{Fun}^L(\mathcal{SH}(k),\on{Sp})[\frac{1}{e}]\\
&\simeq\mathcal{SH}(k)^{\vee}[\frac{1}{e}].
\end{aligned}\]
To see the second and third isomorphisms, we regard both sides as full subcategories of $\on{Fun}^L(\mathcal{SH}(k),\on{Sp})$ and check they are restricted to the same class of functors. A key ingredient here is the following characterization of inverting a number $e$.
\begin{lm}
Let $\mathcal{D}$ be a stable category with sequential colimits. Then $\mathcal{D}[\frac{1}{e}]$ is precisely the full subcategory of $\mathcal{D}$ spanned by those $d\in\mathcal{D}$ such that \[d\xrightarrow{\times e}d\]
is an equivalence.
\end{lm}
\begin{proof}
This follows that multiplication with a number $e$ commutes with taking mapping spectra. Hence for each $c$ and $d$, \[\on{Map}(c,d)\xrightarrow{\times e}\on{Map}(c,d)\]
identifies precomposition with $e$ and postcomposition with $e$.
\end{proof}
As a consequence, we have a well-defined map\[
\mathcal{SH}(k)[\frac{1}{e}]^{\vee}\simeq\mathcal{SH}(k)^{\vee}[\frac{1}{e}]\rightarrow\on{Mot}(k)_{\on{loc}}^{\mathbb{A}^1}[\frac{1}{e}]\]
with a rigidly generated domain. This functor is again monoidal and left adjoint, and we denote it again by $\Phi$ by abuse of notation.

We also note that $\on{Mot}(k)_{\on{loc}}^{\mathbb{A}^1}[\frac{1}{e}]$ has compact unit. In fact, $\on{Mot}(k)_{\on{loc}}^{\mathbb{A}^1}[\frac{1}{e}](1,-)$ is exactly $\on{KH}(-)$ by Proposition \ref{KH}, and $\on{KH}$ preserves filtered colimits. Combining the discussions above, it follows Lemma \ref{same.modules} that\[
\on{Mod}_{\Phi_\ast\Phi^\ast}(\mathcal{SH}(k)[\frac{1}{e}]^\vee)\simeq\on{Mod}_{\Phi_\ast(1)}(\mathcal{SH}(k)[\frac{1}{e}]^\vee).
\]

We will use the following notations:\[\begin{tikzcd}
	{\mathcal{SH}(k)[\frac{1}{e}]^{\vee}} && {\on{Mot}(k)_{\on{loc}}^{\mathbb{A}^1}[\frac{1}{e}]} \\
	\\
	{\on{Mod}_{\Phi_\ast(1)}(\mathcal{SH}(k)[\frac{1}{e}]^{\vee})}
	\arrow["{\Phi^\ast:=\Phi}", shift left=2, from=1-1, to=1-3]
	\arrow[from=1-1, to=3-1]
	\arrow["{\Phi_\ast}", shift left=2, from=1-3, to=1-1]
	\arrow["{\widetilde{\Phi}_\ast}", shift left=2, from=1-3, to=3-1]
	\arrow["{\widetilde{\Phi}^\ast:=\widetilde{\Phi}}", shift left=2, from=3-1, to=1-3].
\end{tikzcd}\]

Over a general qcqs scheme $S$, $\Phi_\ast(1)$ behaves like a dual version of $\on{KGL}^{\mathbb{A}^1}$ in general. In fact, for all $X\in Sm^{\on{fp}}(S)$, we have\[
\begin{aligned}
\mathcal{SH}(S)^{\vee}((\Sigma^\infty_+ X)^{dual},\Phi_\ast(1)) &\simeq \on{Mot}(S)_{\on{loc}}^{\mathbb{A}^1}(\Phi^\ast((\Sigma^\infty_+ X)^{dual}), 1)\\
&\simeq \on{Mot}(S)_{\on{loc}}^{\mathbb{A}^1}(\mathcal{U}(\on{perf}_X),1)\\
&\simeq\on{KH}(\on{perf}_X^{op}).
\end{aligned}\]

And when $\mathcal{SH}$ is rigidly generated, for instance, over a field $k$ and after inverting the exponential characteristic $e$, $\Phi_\ast(1)$ is identified with $\on{KGL}^{\mathbb{A}^1}$ via the canonical equivalence witnessing its self-duality (see 9.2.1 of \cite{GR})\[\begin{aligned}
\dagger:\mathcal{SH}(k)[\frac{1}{e}]&\rightarrow\mathcal{SH}(k)[\frac{1}{e}]^{\vee}\\
X&\mapsto\on{Map}_{\mathcal{SH}(k)[\frac{1}{e}]}(1,X\otimes-).
\end{aligned}
\]
To see this, note first that for smooth projective scheme $X$, $\Sigma^\infty_+X$ is dualizable in $\mathcal{SH}$ as discussed in \cite{Rio}, hence it remains so in the dual category $\mathcal{SH}^{\vee}$. We have the following equivalences of mapping spectra:
\[\begin{aligned}
\mathcal{SH}^\vee((\Sigma^\infty_+X^{dual})^\vee,\dagger\on{KGL}^{\mathbb{A}^1})&\simeq\mathcal{SH}(\dagger(\Sigma^\infty_+X^{dual})^\vee,\on{KGL}^{\mathbb{A}^1})\\
&\simeq\mathcal{SH}(\Sigma^\infty_+X,\on{KGL}^{\mathbb{A}^1})\\
&\simeq\on{KH}(X)\\
&\simeq\mathcal{SH}^\vee((\Sigma^\infty_+X^{dual})^\vee,\Phi_\ast(1).
\end{aligned}\]
The second isomorphism above follows a combination of Remark \ref{dual.comp} and \cite[Lemma 9.2.4]{GR}. Now since $\mathcal{SH}^\vee$ is rigidly generated, it is generated under suspensions, loops, finite colimits and filtered colimits by objects of the form $(\Sigma^\infty_+X)^{dual}$, showing that $\dagger\on{KGL}^{\mathbb{A}^1}\simeq\Phi_\ast(1)$ by Yoneda lemma.

The main result of this section is the fully-faithfulness of $\widetilde{\Phi}$:
\begin{prop}\label{comparison.SH}
Over a field $k$ and after inverting the exponential characteristic $e$, the functor\[
\widetilde{\Phi}:\on{Mod}_{\Phi_\ast(1)}(\mathcal{SH}(k)[\frac{1}{e}]^\vee)\rightarrow\on{Mot}(k)^{\mathbb{A}^1}_{\on{loc}}[\frac{1}{e}]
\]
is fully faithful on the level of mapping spectra.
\end{prop}
\begin{proof}
In the proof, when we use the notation $\Sigma^\infty_+X$, we actually mean the further localization inverting $e$ of $\Sigma^\infty_+X\in\mathcal{SH}(k)$.

We first claim that this is true for objects of the form $\Phi_\ast(1)\otimes(\Sigma^\infty_+ Z)^{dual}$ with $Z$ smooth and projective. Such assumptions on $Z$ provide sufficient dualizibility:
\begin{itemize}
\item $(\Sigma^\infty_+ Z)^{dual}$ is dualizable in $\mathcal{SH}(k)[\frac{1}{e}]^\vee$;
\item $\on{perf}_Z$ is saturated, that is, dualizable in the ambient symmetric monoidal category $\PrL_{cg}\simeq\mathsf{Cat}^{\on{perf}}$. This follows \cite[Proposition 3.31]{Orl}: For $k$-scheme $Z$ of finite type and separated, $\on{perf}_Z$ is smooth (resp. proper) if and only if $Z$ is smooth (resp. proper).
\end{itemize}
Then we have\[\begin{aligned}
&\on{Mod}_{\Phi_\ast(1)}(\mathcal{SH}(k)[\frac{1}{e}]^{\vee})(\Phi_\ast(1)\otimes(\Sigma^\infty_+ X)^{dual},\Phi_\ast(1)\otimes(\Sigma^\infty_+ Y)^{dual})\\
\simeq&\mathcal{SH}(k)[\frac{1}{e}]^{\vee}((\Sigma^\infty_+ X)^{dual},\Phi_\ast(1)\otimes(\Sigma^\infty_+ Y)^{dual})\\
\simeq&\mathcal{SH}(k)[\frac{1}{e}]^{\vee}((\Sigma^\infty_+ X)^{dual}\otimes((\Sigma^\infty_+ Y)^{dual})^\vee,\Phi_\ast(1))\\
\simeq&\on{Mot}(k)^{\mathbb{A}^1}_{\on{loc}}(\Phi^\ast((\Sigma^\infty_+ X)^{dual}\otimes((\Sigma^\infty_+ Y)^{dual})^\vee),1)[\frac{1}{e}]\\
\simeq&\on{Mot}(k)^{\mathbb{A}^1}_{\on{loc}}(\Phi^\ast(\Sigma^\infty_+ X)^{dual}\otimes\Phi^\ast((\Sigma^\infty_+ Y)^{dual})^\vee,1)[\frac{1}{e}]\\
\simeq&\on{Mot}(k)^{\mathbb{A}^1}_{\on{loc}}(\mathcal{U}\on{perf}_X\otimes\mathcal{U}(\on{perf}_Y^{op}),1)[\frac{1}{e}]\\
\simeq&\on{Mot}(k)^{\mathbb{A}^1}_{\on{loc}}(\mathcal{U}\on{perf}_X,\mathcal{U}\on{perf}_Y)[\frac{1}{e}].
\end{aligned}\]

On the other hand, it suffices to prove this. By rigid generation, $\Phi_\ast(1)\otimes(\Sigma^\infty_+ Z)^{dual}$ is dualizable in $\on{Mod}_{\Phi_\ast(1)}(\mathcal{SH}(k)[\frac{1}{e}]^{\vee})$ and generates the whole category under suspensions, loops, finite colimits and filtered colimits. Applying \cite[Proposition 2.7]{Rob}, it reduces to show that $\widetilde{\Phi}$ preserves compact objects, which is true by the compactness of unit in $\on{Mot}(k)_{\on{loc}}^{\mathbb{A}^1}[\frac{1}{e}]$.
\end{proof}

\subsection{Fully-faithfulness of the factored functor, non-$\mathbb{A}^1$-invariant case}\label{comparison.MS}
One of the differences between $\mathcal{MS}$ and $\mathcal{SH}$ is that, even over a field $k$ and after inverting the exponential characteristic, $\mathcal{MS}$ fails to be rigidly generated. In particular, the categories of modules $\on{Mod}_{\Psi_\ast\Psi^\ast}(\mathcal{MS}^\vee)$ and $\on{Mod}_{\Psi_\ast(1)}(\mathcal{MS}^\vee)$ are different in general.

We will show that the functor\[
\widetilde{\Psi}:\on{Mod}_{\Psi_\ast(1)}(\mathcal{MS}(k)[\frac{1}{e}]^\vee)\rightarrow\on{Mot}(k)_{\on{loc}}[\frac{1}{e}]\]
induced as in Section \ref{fac} is not always fully faithful. In fact, as we will see, over a countable field, $\pi_0$ of the mapping spectrum between the affine line and any non-zero ring $R$ is always countable in the left hand side, while it is always uncountable in $\on{Mot}_{\on{loc}}$.
\begin{prop}\label{countability}
Let $k$ be a countable field and $e$ be the exponential characteristic of $k$. Then the full subcategory of $\mathcal{MS}(k)[\frac{1}{e}]^\vee$ spanned by compact objects is countable, in the sense that the $\pi_0$ of all the mapping spectra are countable.
\end{prop}
\begin{proof}
It suffices to show for $\mathcal{MS}(k)$. We first show that $\mathcal{P}sh_{\on{Nis},\on{ebu}}(Sm^{\on{fp}}(k))$ has countable subcategory of compact objects. By \cite[Theorem I.3.3]{NS} and \cite[Proposition 5.6]{BGT}, we see that $\mathcal{P}sh_{\on{Nis},\on{ebu}}(Sm^{\on{fp}}(k))$ is the Verdier quotient of\[\mathcal{A}\rightarrow\mathcal{P}sh(Sm^{\on{fp}}(k)),\]
where $\mathcal{A}$ is spanned by, for all Nisnevich squares and elementary blowup squares\[\begin{tikzcd}
	W && U \\
	\\
	V && X
	\arrow[from=1-1, to=1-3]
	\arrow[from=1-1, to=3-1]
	\arrow[from=1-3, to=3-3]
	\arrow[from=3-1, to=3-3],
\end{tikzcd}\]
the cofibers of\[
j(U)\amalg_{j(W)}j(V)\rightarrow j(X).\]
Moreover, the mapping spectra in $\mathcal{P}sh_{\on{Nis},\on{ebu}}(Sm^{\on{fp}}(k))$ between $\on{L}\circ j(X)$ and $\on{L}\circ j(Y)$ is computed by\[
colim_{Z \in \mathcal{A}_{/ Y}} \operatorname{Map}_{\mathcal{P}sh(Sm^{\on{fp}}(k))}(j(X), \operatorname{cofib}(Z \rightarrow j(Y))),\]
hence is countable as a countable colimit of countable sets.

To pass to $\mathcal{MS}(k)$, we consider the c-spectra model. It directly follows Definition \ref{lax.} that $\on{Sp}^{lax}_{\mathbb{P}^1}(\mathcal{P}sh_{\on{Nis},\on{ebu}}(Sm^{\on{fp}}(k)))$ has countable mapping spectra between compacts. Applying \cite[Theorem I.3.3]{NS} again shows that $\mathcal{MS}(k)$ has countable mapping spectra between compacts, as desired.
\end{proof}

Since the mapping spectra between representables in $\on{Mod}_{\Psi_\ast(1)}(\mathcal{MS}(k)[\frac{1}{e}]^\vee)$ is computed in $\mathcal{MS}(k)[\frac{1}{e}]^\vee$ up to tensoring the target with the $\Psi_\ast(1)$, and the latter one is a $\omega_1$-compact object as a dual version of $\on{KGL}$, we conclude that the mapping spectra between representables in $\on{Mod}_{\Psi_\ast(1)}(\mathcal{MS}(k)[\frac{1}{e}]^\vee)$ is countable as a $\omega_1$-small colimit of countable sets.

On the other hand, for a non-zero $k$-algebra $R$, \cite[Theorem 9.1]{Efi25} shows that\[
\mathbb{W}(R)\simeq\pi_0\on{Map}_{\on{Mot}_{\on{loc}}}(\mathcal{U}\on{perf}_{k[t]},\mathcal{U}\on{perf}_R),\]
where $\mathbb{W}(R)$is the ring of big Witt vectors. The underlying set of $\mathbb{W}(R)$ is $(1+tR[[t]])^\times$. In particular, it is uncountable.
\bibliographystyle{alpha}
\bibliography{Biblio.bib}
\end{document}